\documentclass[reqno,10pt]{amsart}

\usepackage{mathtools}
\usepackage{enumitem}

\usepackage[a4paper,left=35mm,right=35mm,top=30mm,bottom=30mm,marginpar=25mm]{geometry}

\usepackage{amsmath}
\usepackage{amssymb}
\usepackage{amsthm}
\usepackage{eurosym}
\usepackage[dvips]{graphics}
\usepackage{graphicx}
\usepackage{epsfig}
\usepackage[colorlinks=true, pdfstartview=FitV, 
 linkcolor=blue, citecolor=blue, urlcolor=blue]{hyperref}

\usepackage{ifthen}

\usepackage{dsfont}
\usepackage[nocompress, space]{cite}

\usepackage[dvipsnames]{xcolor}

\allowdisplaybreaks

\newcommand{\essinf}{\operatorname{essinf}}
\renewcommand{\div}{\operatorname{div}}

\newcommand{\Nn}{{\mathbb{N}}}

\newcommand{\Tt}{{\mathbb{T}}}

\newcommand{\RR}{{\mathbb R}}

\newcommand{\epsi}{\varepsilon}
\renewcommand{\epsilon}{\varepsilon} %this is cheating by DG

\def\leq{\leqslant}
\def\geq{\geqslant}

\numberwithin{equation}{section}

\newtheoremstyle{thmlemcorr}{10pt}{10pt}{\itshape}{}{\bfseries}{.}{10pt}{{\thmname{#1}\thmnumber{
                        #2}\thmnote{ (#3)}}}

\newtheoremstyle{thmlemcorr*}{10pt}{10pt}{\itshape}{}{\bfseries}{.}\newline{{\thmname{#1}\thmnumber{
                        #2}\thmnote{ (#3)}}}

\newtheoremstyle{defi}{10pt}{10pt}{\itshape}{}{\bfseries}{.}{10pt}{{\thmname{#1}\thmnumber{
                        #2}\thmnote{ (#3)}}}

\newtheoremstyle{remexample}{10pt}{10pt}{}{}{\bfseries}{.}{10pt}{{\thmname{#1}\thmnumber{
                        #2}\thmnote{ (#3)}}}

\newtheoremstyle{remexample*}{10pt}{10pt}{}{}{\bfseries}{.}{10pt}{{\thmname{#1}\thmnote{ (#3)}}}

\newtheoremstyle{ass}{10pt}{10pt}{}{}{\bfseries}{.}{10pt}{{\thmname{#1}\thmnumber{
                        A#2}\thmnote{ (#3)}}}

\theoremstyle{thmlemcorr}
\newtheorem{theorem}{Theorem}
\numberwithin{theorem}{section}
\newtheorem{lemma}[theorem]{Lemma}

\newtheorem{proposition}[theorem]{Proposition}

\newtheorem{assumption}[theorem]{Assumption}

\theoremstyle{thmlemcorr*}
\newtheorem{theorem*}{Theorem}
\newtheorem{lemma*}[theorem]{Lemma}
\newtheorem{corollary*}[theorem]{Corollary}
\newtheorem{proposition*}[theorem]{Proposition}
\newtheorem{problem*}[theorem]{Problem}
\newtheorem{conjecture*}[theorem]{Conjecture}

\theoremstyle{defi}
\newtheorem{definition}[theorem]{Definition}

\newtheorem{problem}{Problem}

\theoremstyle{remexample}
\newtheorem{remark}[theorem]{Remark}

\usepackage{enumitem}

\providecommand{\varitem}{} %to keep LaTeX %quiet
\makeatletter
\newenvironment{hypotheses}[1]
 {\renewcommand\varitem[1]{\item[
 \textbf{#1\arabic{enumi}\rlap{$##1$}.}]%
    \edef\@currentlabel{#1\arabic{enumi}{$##1$}}}%
  \enumerate[label=\bf{({#1\arabic*})}, ref=#1\arabic*]}
 {\endenumerate}
\makeatother

\theoremstyle{remexample*}
\newtheorem{example*}{Example}

\theoremstyle{ass}

\begin{document}
        
        \title[Mean-Field Games with Monotonicity Methods]{Solving Mean-Field Games with Monotonicity Methods in Banach Spaces}

\author{Rita Ferreira}
\address[R. Ferreira]{
        King Abdullah University of Science and Technology (KAUST),
        CEMSE Division, Thuwal 23955-6900, Saudi Arabia.}
\email{rita.ferreira@kaust.edu.sa}
\author{Diogo Gomes}
\address[D. Gomes]{
        King Abdullah University of Science and Technology (KAUST),
        CEMSE Division, Thuwal 23955-6900, Saudi Arabia.}
\email{diogo.gomes@kaust.edu.sa}
\author{Melih \"U\c{c}er}
\address[M. \"U\c{c}er]{
 King Abdullah University of Science and Technology (KAUST),
        CEMSE Division, Thuwal 23955-6900, Saudi Arabia \& 
 Ankara Yildirim Beyazit University (AYBU),
        Ankara, Turkey.
        }
\email{melih.ucer@kaust.edu.sa}

 %Existence problems: global existence, local existence, non-existence
        
\thanks{Funds. }
\date{\today}

\begin{abstract}
This paper develops a unified framework for proving the existence of solutions to stationary first-order mean-field games (MFGs) based on the theory of monotone operators in Banach spaces. We cast the coupled MFG system as a variational inequality, 
overcoming the limitations of prior Hilbert-space approaches that relied on high-order regularization and typically yielded only weak solutions in the monotone operator sense. In contrast, with our low-order regularization, we obtain strong solutions.

Our approach addresses the non-coercivity of the underlying MFG operator through two key regularization strategies. First, by adding a low-order $p$-Laplacian term, we restore coercivity, derive uniform a priori estimates, and pass to the limit via Minty's method.
This establishes, for the first time via monotonicity methods, the existence of \emph{strong solutions} for models with both standard power-growth and singular congestion, with the latter requiring a careful restriction of the operator's domain. Second, for Hamiltonians with only minimal growth hypotheses, we regularize the Hamiltonian itself via infimal convolution to prove the existence of \emph{weak solutions}.  

Our Banach-space framework unifies and extends earlier existence results. By avoiding high-order smoothing, it not only provides a more direct theoretical path but is also ideally suited for modern numerical algorithms.
\end{abstract}

\maketitle
 
\section{Introduction}

Mean-Field Games (MFG) theory, introduced by Lasry and Lions \cite{ll1,ll2} and by Caines, Huang, and Malham\'{e} \cite{Caines2}, has become an indispensable tool for modeling strategic interactions in large populations of anonymous rational agents. With applications ranging from modeling price formation in financial markets to steering pedestrian flows in urban planning, one of the theory's central challenges lies in proving the existence of an equilibrium. Our approach relies on \emph{monotonicity methods}, which exploit the structure of the underlying equations---specifically, the property of an operator $A$ satisfying $\langle A(x)-A(y), x-y \rangle \geq 0$---to guarantee that such equilibria exist.

The equilibrium of a stationary MFG is described by a coupled system of nonlinear partial differential equations (PDEs), typically a Hamilton--Jacobi equation and a transport or Fokker--Planck equation. Historically, research on this system has advanced along two main threads: the existence of classical (smooth) solutions \cite{GM, MR3113415, GPatVrt, GPM2, GPM3, Gomes2015b, Gomes2016c, PV15, cirant3, CirGo20} and the existence of weak solutions in Sobolev spaces, often obtained through variational methods \cite{porretta, porretta2, bocorsporr, cgbt, cardaliaguetMeanFieldGames2014, MR3691806, AMFS}.

The analysis is largely shaped by the presence or absence of a \emph{monotonicity} condition. While some results exist for non-monotone games, they are generally local in nature, relying on contraction arguments valid only over short time intervals or on other smallness conditions \cite{ciranttonon,MR4064664, Ambrose}. In contrast, most of the literature assumes a monotonicity structure. This property was not only central to the original uniqueness proofs of Lasry and Lions \cite{lasryMeanFieldGames2007a} but has also inspired successful numerical schemes \cite{almulla2017two, GJS2}.
{However, beyond its role in uniqueness, monotonicity has also been used to build general existence results \cite{FG2, FGT1, FeGoTa21} and to define and study monotone solutions for the MFG master equation \cite{bertucciMonotoneSolutionsMean2023}. In the present work, our Banach-space methods are used only to establish existence; uniqueness requires a separate analysis and is not pursued here.}

Despite these advances, existence proofs based on monotonicity methods \cite{FG2, FGT1, FeGoTa21} face notable limitations. They are typically set in Hilbert spaces (e.g., $H^{2k} \times H^{2k}$), necessitating the addition of high-order regularizing operators (such as $(\cdot\, + \Delta^{2k} \cdot)$ for  $k\in\Nn$ sufficiently large) to ensure coercivity and continuous embeddings into $L^\infty$. These high-order perturbations can be technically cumbersome and   fundamentally limit the regularity analysis: the limit equations hold only in the duals of high-order Sobolev spaces, yielding only weak solutions in the abstract sense of monotone operators (see Definition~\ref{def:weak}), in contrast to the strong pointwise solutions known to exist for variational problems \cite{cardaliaguetMeanFieldGames2014}. 

This reveals a significant gap: a unified Banach space framework that matches the integrability of the function spaces directly to the Hamiltonian's growth, 
{
thereby avoiding high-order smoothing and yielding strong solutions for power-growth and singular-congestion models, while still covering minimal-growth Hamiltonians at the weak-solution level.
}
Such a first-order framework is not only theoretically cleaner but also highly advantageous for computation. The low-order $p$-Laplacian regularizations we introduce preserve the local structure of the PDEs, making them ideally suited for the design and implementation of modern monotone inclusion and splitting algorithms \cite{nurbekyanNoteConvergenceMonotone2023, nurbekyanMonotoneInclusionMethods2024}.

This paper develops a unified framework rooted in the theory of monotone operators in Banach spaces, addressing these challenges. Our approach, which avoids the high-order regularizations found in prior works, is built on three key technical innovations. First, we recast the entire MFG system as a variational inequality for a single monotone operator. Second, we introduce a low-order $p$-Laplacian regularization; this technique is not only critical to ensuring coercivity and obtaining strong solutions for a broad class of problems, but its structure is also suitable for numerical methods. Third, for models under only minimal growth conditions, we employ infimal convolution to regularize the Hamiltonian itself, allowing us to establish the existence of weak solutions where other methods may not apply.

As for the concrete applications, we primarily study the existence of solutions to the following first-order, local, stationary mean-field game with periodic boundary conditions; that is, we address the study of \eqref{smfg} in the $d$-dimensional torus $\Tt^d$, $d\in\Nn$.
\begin{problem}\label{prob.smfg}
    Given a Hamiltonian $H\colon \Tt^d\times \RR^d \times \RR^+ \to \RR$ and a potential $V\colon \Tt^d\to \RR$, establish existence of solutions $(m,u)$ with $m\colon\Tt^d\to \RR^+_0$ and $u\colon\Tt^d\to \RR$ to the system
    \begin{equation}\label{smfg}
        \begin{aligned}
                \begin{cases}
                        -u- H(x,Du,m)+ V(x) =0,\\
                        m - \div(m D_pH(x, Du,m)) -1 =0. 
                \end{cases}
        \end{aligned}
    \end{equation}
\end{problem}
    The presence of the term $-u$ in the Hamilton--Jacobi equation corresponds to a strictly positive discount rate in the agents' optimal control problem, which is a standard formulation for discounted stationary mean-field games (see \cite{FGT1, FGTchapter2026} or \cite{MR4175148},  for instance).

Our key structural assumption on the Hamiltonian~$H$ is a certain monotonicity criterion (Assumption~\ref{onH.monotone}) that enables us to use well-established monotone operator methods. Each theorem in this paper requires this assumption, in addition to certain growth and coercivity conditions that vary between the theorems. These theorems state existence of weak and strong solutions as defined next, and further discussed in Subsection~\ref{subs:solutions}. Note that merely defining the solutions require a certain minimal amount of regularity on the data, which we make precise in Subsection~\ref{subs:assumptions} and assume whenever needed.

\begin{definition}[Strong solution]\label{def:strong} Suppose that $H$ satisfies the regularity conditions (\ref{a:regularityH}), (\ref{a:Hat0}), and (\ref{a:mDpHat0}) and that $V$ is measurable. We say that $(m,u)\in L^1(\Tt^d)\times W^{1,1}(\Tt^d)$ is a \emph{strong solution} to Problem~\ref{prob.smfg} if $m\geq 0$, $mD_pH(x,Du,m)\in L^1(\Tt^d; \RR^d)$, and the two equations in \eqref{smfg} hold in the following sense:
\begin{enumerate}
    \item The Hamilton--Jacobi equation holds in the sense that
    \begin{equation}\label{def:hjb-strong}
        \begin{aligned}
        & u + H(x,Du,m) - V(x) \leq 0 \qquad \text{a.e.~in } \Tt^d, \\
            & u + H(x,Du,m) - V(x) = 0 \qquad \text{a.e.~in the set } \{m>0\}.
        \end{aligned}
    \end{equation}
    \item The transport equation holds in the sense that
    \begin{equation}\label{def:transport-strong}
        \int_{\Tt^d} \big[(m-1)\varphi + mD_pH(x,Du,m)\cdot D\varphi\big] \,dx = 0
    \end{equation}
    for all $\varphi\in C^\infty(\Tt^d)$.
\end{enumerate}
Note that, in the set $\{m=0\}$, the expressions $H(x,Du,m)$ and $mD_pH(x,Du,m)$ must be evaluated in the sense of (\ref{a:Hat0}) and (\ref{a:mDpHat0}).  
\end{definition}
\begin{remark}\label{rmk:onDefStrong}
We note that the Hamilton--Jacobi equation in the preceding definition is written  with signs reversed compared to \eqref{smfg}; this mathematically equivalent formulation allows for a more direct comparison with the standard viscosity solutions literature for Hamilton–Jacobi (HJ) equations.

 We further observe that the regularity required for strong solutions constitutes the minimal functional framework necessary for the definition.  However, our main theorems establish existence of such solutions in spaces with stronger integrability. 
\end{remark}

\begin{definition}[Weak Solution]\label{def:weak}
Suppose that $H$ satisfies Assumption~\ref{onH.monotone} and that $V$ is measurable. We say that $(m,u)\in L^1(\Tt^d)\times W^{1,1}(\Tt^d)$ is a \emph{weak solution} to Problem~\ref{prob.smfg} if $m\geq 0$ and
\begin{equation}\label{def:weak-solution-concept}
\begin{aligned}
0 \leq& \int_{\Tt^d}
\Big[ \big(  -\upsilon- H(x, D\upsilon,\mu) + V(x)  \big)(\mu
 -m)\\
&\qquad\quad+ \mu D_p H(x, D\upsilon,\mu) \cdot(D\upsilon - Du)  + (\mu-1)( \upsilon-u) \Big]
\,dx
\end{aligned}
\end{equation}
% \begin{equation}
    %     \begin{aligned}
    %         \int_{\Tt^d} \big[(\mu-1)(\upsilon-u) + \mu D_pH(x,D\upsilon,\mu)\cdot (D\upsilon-Du)\big]\, \dx \geq \qquad\quad\\
    %         \int_{\Tt^d} (\upsilon + H(x,D\upsilon,\mu) - V(x))(\mu-m) \, \dx
    %     \end{aligned}
    % \end{equation}
for all $(\mu,\upsilon)\in L^\infty(\Tt^d)\times W^{1,\infty}(\Tt^d)$ with $\essinf\mu > 0$.
\end{definition}

Our framework yields the following three main existence theorems, presented here in
order of increasing technical difficulty. These results all rely on the standing monotonicity condition, Assumption~\ref{onH.monotone}, and on a corresponding growth assumption from a hierarchy of progressively weaker conditions detailed in Subsection~\ref{subs:assumptions}.

First, for models with standard power-growth, we prove the existence of strong solutions (Theorem~\ref{thm:main-simple.power.growth}, proved in Section \ref{sec:proof-power-growth}). The central challenge, which our method is designed to overcome, is the lack of coercivity of the underlying MFG operator.

\begin{theorem}\label{thm:main-simple.power.growth}
        Assume that $H:\Tt^d\times \RR^d \times \RR^+ \to \RR$ satisfies Assumption~\ref{onH.monotone} and Assumption~\ref{onH.powergrowth}. Let $\bar\beta$ and $\bar\gamma$ be given by \eqref{eq:gamma}, and assume that $V\in L^{\bar\beta'}(\Tt^d)$.  Then, there exists a strong solution $(m,u)\in L^{\bar \beta}(\Tt^d) \times W^{1,\bar\gamma}(\Tt^d)$ to Problem~\ref{prob.smfg} in the sense of Definition~\ref{def:strong}.%with $m\geq0$        
\end{theorem}

We then extend this result to models with singular congestion terms, again establishing the existence of strong solutions (Theorem~\ref{thm:main-congestion.growth}, proved in Section \ref{sec:proof-congestion}). The main difficulty in this setting is that the operator is not even well-defined on the standard function spaces; the proof hinges on a careful and highly non-trivial construction of a restricted convex domain.

\begin{theorem}\label{thm:main-congestion.growth}
Assume that $H:\Tt^d\times \RR^d \times \RR^+ \to \RR$ satisfies Assumption~\ref{onH.monotone} and Assumption~\ref{onH.congestiongrowth}. Let $\bar\beta$ and $\bar\gamma$ be given by \eqref{eq:gamma}, and assume that $V\in L^{\bar\beta'}(\Tt^d)$.  Then, there exists a strong solution $(m,u)\in L^{\bar \beta}(\Tt^d) \times W^{1,\bar\gamma}(\Tt^d)$ to Problem~\ref{prob.smfg} in the sense of Definition~\ref{def:strong}.
\end{theorem}

Finally, we address Hamiltonians satisfying only minimal growth assumptions and prove the existence of weak solutions (Theorem~\ref{thm:main-weak.power.growth}, proved in Section \ref{sec:proof-weak-growth}). This most challenging case requires a different technique, infimal convolution of the Hamiltonian, to compensate for the lack of controlling growth assumptions on the Hamiltonian.

\begin{theorem}\label{thm:main-weak.power.growth} 
    Assume that $H:\Tt^d\times \RR^d \times \RR^+ \to \RR$ satisfies Assumption~\ref{onH.monotone} and Assumption~\ref{onH.weakgrowth}. Let $\bar\beta$ and $\bar\gamma$ be given by \eqref{eq:gamma}, and assume that $V\in L^{\bar\beta'}(\Tt^d)$. Then, there exists a weak solution $(m,u)\in L^{\bar \beta}(\Tt^d) \times W^{1,\bar\gamma}(\Tt^d)$ to Problem~\ref{prob.smfg} in the sense of Definition~\ref{def:weak}.
\end{theorem}

To illustrate the scope of our results, we give three prototypical Hamiltonians covered by our assumptions. These are simplified examples; our theory applies to more general, non-separable forms, including the power-growth case below.
\begin{itemize}
    \item \textbf{Power‐growth case:} An example covered by Theorem~\ref{thm:main-simple.power.growth} is 
    \[
        H(x,p,m)=a(x)\,|p|^\alpha - b(x)\,m^\beta,
    \]
    where $\alpha>1$, $\beta>0$, and the functions $a,b\in L^\infty(\Tt^d)$ satisfy $\essinf a>0$ and $\essinf b>0$.

    \item \textbf{Congestion case:} An example covered by Theorem~\ref{thm:main-congestion.growth} is
    \[
        H(x,p,m)=a(x)\,\frac{|p|^{\alpha\big(1+\tfrac\tau\beta\big)}}{m^\tau} - b(x)\,m^\beta,
    \]
    with the same conditions on $\alpha,\beta,a,b$ as above, and with a congestion exponent $\tau\in[0,1]$.

    \item \textbf{Weak‐growth case:} An example covered by Theorem~\ref{thm:main-weak.power.growth} is
    \[
        H(x,p,m)=g(x)\,h(p,m) + a(x)\,|p|^\alpha - b(x)\,m^\beta,
    \]
    where $\alpha,\beta,a,b$ are as before, $g\in L^\infty(\Tt^d)$ with $g\geq 0$, and $h:\RR^d\times\RR^+\to\RR^+_0$ is any $C^1$ function satisfying the monotonicity condition \eqref{hmon} but no further growth bounds.
\end{itemize}

The remainder of this paper is organized as follows. In Section~\ref{sect:prel}, we present the mathematical setting, including our standing assumptions and key results from monotone operator theory. The subsequent three sections are devoted to the proofs of our main results: Section~\ref{sec:proof-power-growth} addresses the power-growth case, Section~\ref{sec:proof-congestion} treats the congestion case, and Section~\ref{sec:proof-weak-growth} establishes the result for minimal growth conditions.

\section{Preliminaries}\label{sect:prel}

This section provides the essential preliminaries for our main results. After a brief discussion of the notation, 
we
outline, in Section~\ref{subs:assumptions}, the assumptions on the Hamiltonian $H$; these establish key analytical properties like monotonicity and controlled growth that are vital for our proofs.
Building upon these assumptions, in Section~\ref{estham}, we then derive several  estimates for the Hamiltonian which are direct consequences of these conditions and will be instrumental in our subsequent analysis. 
Next, in Section~\ref{subs:monotone}, we review core concepts from monotone operator theory and conclude, in Section~\ref{subs:solutions}, by discussing the specific notions of strong and weak solutions employed in this paper.

\subsection{Notation}

Here, and throughout this manuscript,  the symbols $\RR$, $\RR^+$, $\RR^+_0$, $\RR^d$, and $\Tt^d$, denote the set of real numbers, the set of strictly positive real numbers, the set of non-negative real numbers, the $d$-dimensional Euclidean space, and the $d$-dimensional torus, respectively. A generic point on $\Tt^d$ is denoted by $x$, and the Sobolev derivative of a function $f\colon \Tt^d\to \RR$ is denoted by $Df\colon \Tt^d\to \RR^d$. For the Hamiltonian, $H\colon \Tt^d\times\RR^d\times \RR^+ \to \RR$, the generic input variable is $(x,p,m)$; as such, the derivative of $H$ with respect to the middle variable is denoted by $D_pH\colon \Tt^d\times \RR^d\times \RR^+ \to \RR^d$, even when the middle variable is substituted by another expression.

\begin{remark}[On constants depending only on the problem data]\label{rmk:onC} 
We adopt the standard PDE convention of reusing the letter \(C\) for constants that may change between estimates. Typically, these constants depend only on problem data (e.g., the various constants in the assumptions).  In particular, the form of any such inequality must be so that whenever $C$ is replaced by a greater value, the inequality remains valid.
\end{remark}

\subsection{Assumptions}\label{subs:assumptions}
In our analysis, we do not require the maps $(x,p,m)\mapsto H(x,p,m)$ or $(x,p,m)\mapsto mD_pH(x,p,m)$ to be defined at $m=0$.
However, because strong solutions may vanish on sets of positive measure, we need to give sense to expressions involving these maps when $m=0$.
The following mild assumptions ensure that $H(x,p,0)$ and  $mD_pH(x,p,m)$ remain well-defined even when  $m=0$.
They also specify the standing regularity on the Hamiltonian.

\begin{hypotheses}{H}
\item \label{a:regularityH} 
The Hamiltonian, $H\colon \Tt^d\times \RR^d \times \RR^+ \to \RR$, is a measurable function such that the map $(p,m) \mapsto H(x,p,m)$ and the $p$-gradient map $(p,m) \mapsto D_pH(x,p,m)$ are continuous on $\RR^d \times \RR^+$ for a.e.~$x\in\Tt^d$. \smallskip

\item \label{a:Hat0} 
For  a.e.~$x\in\Tt^d$ and all $p\in\RR^d$,  the limit 
\[\lim_{m\to 0^+} H(x,p,m)\] 
exists in $\RR \cup \{\infty\}$. In this case, we may consider the  $(\RR \cup \{\infty\})$-valued extension of $H$, $\widetilde H: \Tt^d\times \RR^d \times \RR^+_0 \to \RR \cup \{\infty\}$ defined by
\begin{equation*}%\label{eq:contH0}
    \begin{aligned}
        \widetilde H(x,p,m) = 
        \begin{cases}
            H(x,p,m) & \text{if } m>0\\
            \displaystyle\lim_{m\to 0^+} H(x,p,m) & \text{if } m=0.
        \end{cases}  
    \end{aligned}
\end{equation*}
To simplify the exposition, by abusing notation, we denote by the same letter, $H$, the extension above. In particular, we write $H(x,p,0)$ in place of $\widetilde H(x,p,0)$. \smallskip

\item  \label{a:mDpHat0}
For  a.e.~$x\in\Tt^d$ and all $p\in\RR^d$, the limit
\[
\lim_{m\to 0^+} m D_p H(x,p,m)
\]
exists in $\RR^d$; moreover, for the map $J: \Tt^d\times \RR^d \times \RR^+_0 \to \RR^d$ defined by
\begin{equation*}%\label{eq:contmDpH0}
    \begin{aligned}
        J(x,p,m) = 
        \begin{cases}
            m D_p H(x,p,m) & \text{if } m>0,\\
            \lim_{m\to 0^+} m D_pH(x,p,m) & \text{if } m=0,
        \end{cases}  
    \end{aligned}
\end{equation*}
the function $J(x,\cdot,\cdot)$ is continuous on $\RR^d\times \RR^+_0$ for a.e.~$x\in\Tt^d$. For notational convenience, we write $mD_p H$ in place of $J$, even when $m=0$.  For instance, given $(\mu,\upsilon)\in L^1(\Tt^d;\RR^+_0) \times W^{1,1}(\Tt^d)$, we write  $\mu D_pH(x,D\upsilon,\mu)$ in place of $J(x,D\upsilon,\mu)$. 
\end{hypotheses}

Throughout this manuscript, we adopt the following standing assumption on the Hamiltonian, encoding a monotonicity property for the MFG in \eqref{smfg} that enables the use of monotone operator techniques.
This is the standard assumption introduced in \cite{ll2} to establish the uniqueness of solutions. 

\begin{assumption}
        \label{onH.monotone}
    $H$ satisfies (\ref{a:regularityH}), and the following monotonicity condition holds for a.e.~$x\in\Tt^d$ and for all $(p_1,m_1)$, $(p_2,m_2) \in\RR^d\times\RR^+$:
        \begin{equation}
    \label{hmon}
                \begin{aligned}
                        &\big(- H(x,p_1,m_1) +  H(x,p_2,m_2)\big) (m_1 - m_2)\\ &\quad + 
                        \big( m_1 D_p H(x,p_1,m_1) - m_2 D_p H(x,p_2,m_2)\big) \cdot (p_1-p_2) \geq 0.
                \end{aligned}
        \end{equation}
\end{assumption}

\begin{remark}\label{rmk:cxty}
Suppose that $H\colon \Tt^d\times \RR^d \times \RR^+ \to \RR$ satisfies Assumption~\ref{onH.monotone}.
It follows by taking $p_1=p_2$ in \eqref{hmon} that for a.e.~$x\in\Tt^d$ and for all $p\in\RR^d$, the function $H(x,p,\cdot)$ is nonincreasing in $\RR^+$. Similarly, it follows by taking $m_1=m_2$ in \eqref{hmon} that for a.e.~$x\in\Tt^d$ and for all $m\in\RR^+$, the function $H(x,\cdot,m)$ is convex in $\RR^d$.

Then, the monotonicity of $H(x,p,\cdot)$ implies (\ref{a:Hat0}): %the extension $H\colon \Tt^d\times \RR^d \times \RR^+_0 \to \RR\cup\{\infty\}$ described in Subsection~\ref{sec:note} exists:
\[H(x,p,0) = \lim_{m\to 0^+} H(x,p,m) = \sup_{m>0} H(x,p,m).\]
Moreover, by the given supremum identity and by the convexity of $H(x,\cdot,m)$, the map $p\mapsto H(x,p,0)$ is lower semicontinuous and convex for a.e.~$x\in\Tt^d$.

When $H$ satisfies \eqref{a:mDpHat0}, then \eqref{hmon} holds for all $(p_1,m_1)$, $(p_2,m_2) \in \RR^d\times\RR^+_0$,  provided
$H(x,p_1,0)$ and $H(x,p_2,0)$ are finite.
\end{remark}

Next, we introduce the additional assumptions that accompany Assumption~\ref{onH.monotone} in the main theorem statements. These assumptions (Assumptions~\ref{onH.powergrowth}, \ref{onH.congestiongrowth}, and \ref{onH.weakgrowth}) are mainly {growth- and coercivity-type conditions} on the Hamiltonian $H$. They fulfill several indispensable roles throughout the analysis:
\begin{itemize}
    \item \textbf{Ensuring Coercivity}: A primary function of these conditions, particularly the lower bounds on $H$ with respect to the momentum $p$ and the terms describing the behavior of $H$ with respect to the density $m$ is to establish the coercivity of the MFG operator. Coercivity is key for applying the abstract existence theorem for monotone operators, Theorem~\ref{thm:monotone.abstract-weak}.
    
    \item \textbf{Guaranteeing Well-Posedness and Integrability}: The growth conditions, including upper bounds on the partial derivative $D_pH$ and on $H$ itself, are crucial for ensuring that all terms in the MFG operator are well-defined. This involves guaranteeing that expressions like $H(x,Du,m)$ and $mD_pH(x,Du,m)$ have the required integrability based on the regularity of $m$ and $u$.

    \item \textbf{Enabling Uniform A Priori Estimates}: These assumptions are fundamental for deriving uniform a priori estimates for the solutions $(m_\epsi, u_\epsi)$ of the appropriate regularized MFG; for example, the operator defined in \eqref{defAepsi} is used to derive the estimates in Theorem~\ref{thm:apriori-simple.power.growth}. These estimates, which must be independent of the regularization parameter $\epsi$, are critical for the compactness arguments used in the proofs. 
\end{itemize}
The hierarchy presented by these assumptions (see Remark~\ref{rem:hierarchy}), moving from standard power-growth (Assumption~\ref{onH.powergrowth}) to less restrictive conditions (Assumption~\ref{onH.weakgrowth}), aims to broaden the applicability of the existence theory to a wider range of MFG models. For example,  Assumption~\ref{onH.congestiongrowth}, is specifically designed to accommodate Hamiltonians that model congestion effects. In such cases, the Hamiltonian may depend on the density $m$ in a possibly singular way (e.g., terms like $m^{-\tau}$).
 Finally, dealing with {weaker growth conditions} as in Assumption~\ref{onH.weakgrowth} requires more sophisticated techniques, such as the Hamiltonian regularization via infimal convolution detailed later in the paper. This typically comes at the cost of establishing the existence of weak solutions (as in Theorem~\ref{thm:main-weak.power.growth}), rather than the strong solutions obtained under more restrictive growth hypotheses.

\begin{assumption}
        \label{onH.powergrowth}      
    $H$ satisfies (\ref{a:regularityH}), and there exist $C>0$, $\alpha>1$, and $\beta>0$ such that the following estimates hold  for a.e.~$x\in\Tt^d$ and for all $(p,m)\in\RR^d\times\RR^+$:
\begin{align}
            &H(x,0,m) \leq -\frac{1}{C}m^{\beta} + C,
            \label{eq:assH.upper.simple} \\
            & |D_pH(x,p,m)| \leq C\big( |p|^{\alpha-1} +  m^{\beta-\frac{\beta}{\alpha}}+1\big),\label{eq:assH.DpH.upper.simple}\\
                        &  H(x,p,m) \geq \frac{1}{C}|p|^{\alpha} - C(  m^{\beta}+1).\label{eq:assH.lower.simple}  
\end{align} 
\end{assumption}

\begin{assumption}\label{onH.congestiongrowth}
        $H$ satisfies (\ref{a:regularityH}) and (\ref{a:mDpHat0}), and there exist $C>0$, $\alpha>1$, $\beta > 0$, and $\tau\in [0,1]$ such that the following estimates hold for a.e.~$x\in\Tt^d$ and for all $(p,m)\in\RR^d\times\RR^+$: 
        \begin{align}
        &H(x,0,m)\leq -\frac{1}{C}m^{\beta} + C, \label{eq:assH.upper.congestion}\\
                & |D_pH(x,p,m)| \leq C\left(\frac{1}{m} + \frac{1}{m^\tau}\right)|p|^{{\alpha(1+\frac{\tau}{\beta})}-1} +  C\left(m^{\beta-\frac{\beta}\alpha}+\frac{1}{m}\right),\label{eq:assH.DpH.upper.congestion}\\
                & H(x,p,m) \geq  \frac{1}{C}\left(\frac{1}{m^\tau + 1} \right)|p|^{\alpha(1+\frac{\tau}{\beta})} - C(m^{\beta}+1). \label{eq:assH.lower.congestion}
    \end{align}
\end{assumption}

\begin{assumption}\label{onH.weakgrowth}
    $H$ satisfies (\ref{a:regularityH}), and there exist $C>0$, $\alpha>1$, and $\beta > 0$ such that the following estimates hold for a.e.~$x\in\Tt^d$ and for all $(p,m)\in\RR^d\times\RR^+$:
        \begin{align}
        & H(x,0,m) \leq -\frac{1}{C}m^{\beta} + C, \label{eq:assH.upper.weak} \\
        & |D_pH(x,p,m)| \leq f(p,m) \text{ for some continuous } f\colon\RR^d\times\RR^+\to \RR,
        %\sup_{(p,m)\in K}|D_pH(x,p,m)| \in L^\infty (\Tt^d) \text{ for any compact } K\in\RR^d\times \RR^+
        \label{eq:assH.DpH.upper.weak}\\
        & H(x,p,m) \geq \frac{1}{C}|p|^{\alpha} - C(m^{\beta}+1). \label{eq:assH.lower.weak}
        \end{align}
\end{assumption}

Based on the parameters in the Assumptions~\ref{onH.powergrowth},~\ref{onH.congestiongrowth}, and~\ref{onH.weakgrowth}, we introduce the following notation:
\begin{equation}\label{eq:gamma}
        \bar\beta = \beta + 1 \quad \hbox{and}\quad 
        \bar\gamma = \alpha\frac{\beta+1}{\beta}.
\end{equation}
Moreover, $\bar\beta'$ and $\bar\gamma'$ denote the conjugate exponents as usual. As such, $\bar\gamma = \alpha\bar\beta'$. %Note that $\bar\gamma > \bar\beta'$ and, equivalently, $\bar\beta > \bar\gamma'$.

\begin{remark}\label{rmk:implied-a}
    First, the bound \eqref{eq:assH.DpH.upper.simple} in Assumption~\ref{onH.powergrowth} implies (\ref{a:mDpHat0}) because the function
    \[J(x,p,m) = 
        \begin{cases}
            m D_p H(x,p,m) & \text{if } m>0\\
            0 & \text{if } m=0
        \end{cases}\]
    is then continuous in $(p,m)$ for a.e.~$x\in\Tt^d$.

    Secondly, the bound \eqref{eq:assH.DpH.upper.congestion} in Assumption~\ref{onH.congestiongrowth} implies
    \[|mD_pH(x,p,m)| \leq C(1 + m^{1-\tau})|p|^{{\alpha(1+\frac{\tau}{\beta})}-1} +  C(m^{\beta+1-\frac{\beta}\alpha}+1),\]
    which then holds for all $(p,m)\in\RR^d\times\RR^+_0$ when the left-hand side is interpreted in the sense of~(\ref{a:mDpHat0}).
    
    Finally, when either Assumption~\ref{onH.powergrowth} or Assumption~\ref{onH.congestiongrowth} hold together with (\ref{a:Hat0}), the corresponding bounds in \eqref{eq:assH.upper.simple} and \eqref{eq:assH.lower.simple} or \eqref{eq:assH.upper.congestion} and \eqref{eq:assH.lower.congestion} hold even when $m=0$.
\end{remark}

\begin{remark}[Hierarchy of the assumptions]\label{rem:hierarchy}
    As mentioned in the Introduction, Assumption~\ref{onH.powergrowth} implies the $\tau=0$ case of Assumption~\ref{onH.congestiongrowth}, while Assumption~\ref{onH.congestiongrowth} implies Assumption~\ref{onH.weakgrowth}. These implications are mostly straightforward, with the only subtlety being that \eqref{eq:assH.lower.congestion} implies \eqref{eq:assH.lower.weak}, possibly with a different value of $C$. However, this implication is easily justified using the following Young-type inequality:
\[\frac{|p|^{\alpha(1+\frac{\tau}{\beta})}}{m^\tau+1} + (m^\beta+1) \geq |p|^\alpha\frac{(m^\beta+1)^{\frac{\tau}{\beta+\tau}}}{(m^\tau+1)^\frac{\beta}{\beta+\tau}} \geq \frac{1}{2^{\frac{\beta}{\beta+\tau}}}|p|^\alpha.\]
\end{remark}

\subsection{Estimates on the Hamiltonian} \label{estham}
Here, we establish some uniform bounds on the Hamiltonian resulting from the assumptions introduced in the preceding section. We begin by showing that convexity as in Remark~\ref{rmk:cxty} and an upper bound as in \eqref{eq:assH.upper.simple}, \eqref{eq:assH.upper.congestion}, or \eqref{eq:assH.upper.weak} yield a lower bound on 
the Lagrangian (see Remark \ref{laglowerbound}).

\begin{lemma}\label{lem:A1+anyother}
Assume that $H\colon \Tt^d\times \RR^d \times \RR^+ \to \RR$ satisfies Assumption~\ref{onH.monotone}. Suppose further that there exist $\beta>0$ and $C>0$ such that for a.e.~$x\in\Tt^d$ and for all $m\in\RR^+$, one has
\begin{equation}\label{eq:common}
    H(x,0,m) \leq -\frac{1}{C}m^{\beta} + C.
\end{equation}
Then,  we have  for a.e.~$x\in\Tt^d$ and for all $(p,m)\in\RR^d\times\RR^+$ that
    \begin{equation}\label{eq:DpHdotpminusH-estimate1}
        D_pH(x,p,m)\cdot p - H(x,p,m) \geq \frac{1}{C} m^{\beta} - C.
    \end{equation}  
\end{lemma}

\begin{proof}
    By Remark~\ref{rmk:cxty}, the map $p\mapsto H(x,p,m)$ is convex at any fixed $(x,m)$; hence, 
    \[H(x,0,m)\geq H(x,p,m) - D_pH(x,p,m)\cdot p.\]
    Thus, we conclude \eqref{eq:DpHdotpminusH-estimate1} by combining the preceding inequality with \eqref{eq:common} and flipping the sign of both sides. %multiplying both sides by $(-1)$.  
\end{proof}

\begin{remark}
\label{laglowerbound}
The expression $D_pH(x,p,m)\cdot p - H(x,p,m)$ in \eqref{eq:DpHdotpminusH-estimate1} represents the Lagrangian written in terms of the momentum coordinate $p$. Indeed, in the context of optimal control, the Lagrangian $L(x,v,m)$ and Hamiltonian $H(x,p,m)$ are related via the Legendre transform, where $v = -D_pH(x,p,m)$ is the velocity and $p$ is the momentum. The inequality \eqref{eq:DpHdotpminusH-estimate1} thus provides a uniform lower bound on the Lagrangian, which is crucial for establishing coercivity in the subsequent analysis. 
The conclusion of the preceding lemma can be improved under Assumption~\ref{onH.powergrowth} to give a lower bound which is similar but in terms of powers of $p$ and $m$. Although we do not use it, we state it in \eqref{eq:alternative2-simple} of Lemma~\ref{lem:lower.simple} below for its independent interest. 
\end{remark}

Next, we first establish an upper bound on the Hamiltonian based on Assumption~\ref{onH.powergrowth}, which is the counterpart of the lower bound in \eqref{eq:assH.lower.simple}.

\begin{lemma}[{cf.~\cite[Lemmas~2.3 and~2.4]{ADFGU}}]\label{lem:ADFGU}
Assume that $H:\Tt^d\times \RR^d \times \RR^+\to \RR$ satisfies Assumption~\ref{onH.powergrowth}. Then, there exists a constant, $C>0$, for which the following bound holds  for  a.e.~$x\in\Tt^d$ and for all $(p,m)\in\RR^d\times\RR^+$:
\begin{equation}\label{eq:boundsH}
        \begin{aligned}
                 H(x,p,m) \leq C(|p|^\alpha +1) - \frac{1}{C}m^{\beta}.
        \end{aligned}
\end{equation}
Moreover, $mD_pH(x,Du,m) = 0$ in the set $\{x\colon m=0\}$ (in the sense of Remark~\ref{rmk:implied-a}) and,  for the parameters $\bar\beta>1$ and  $\bar\gamma>1$ defined in  \eqref{eq:gamma}, it holds that
\begin{equation}\label{eq:boundsmDpH}
        \begin{aligned}
                mD_pH(\cdot, Du, m) \in L^{\bar\gamma'}(\Tt^d;\RR^d)\enspace \hbox{ whenever } \enspace m\in L^{\bar \beta}(\Tt^d;\RR^+_0) \hbox{ and } u\in W^{1,\bar\gamma}(\Tt^d).
        \end{aligned}
\end{equation}
\end{lemma}

\begin{proof}
    A similar result has  been established in \cite[Lemma~2.3 with $\tau=0$]{ADFGU} and \cite[Lemma~2.4]{ADFGU} under slightly different assumptions. For this reason, we include  the details here. 

    We first prove \eqref{eq:boundsH}. Using first the fundamental theorem of calculus, and then the bounds \eqref{eq:assH.upper.simple} and \eqref{eq:assH.DpH.upper.simple} to estimate the resulting  terms on the right-hand side,  we conclude  that  
    \begin{equation}\label{eq:case1-lb1}
        \begin{aligned}
            H(x,p,m) &= H(x,0,m) + \int_0^{1} D_p H(x,pt,m) \cdot p \, dt  \\
            & \leq -\frac{1}{C} m^\beta +C + C|p|^\alpha + C|p|\Big(m^{\beta - \frac\beta\alpha} +1\Big).
        \end{aligned}
    \end{equation}
On the other hand, Young's inequality yields for any $\theta>0$ that
\begin{equation}\label{eq:Young.simple.case}
    \begin{aligned}
   |p|(m^{\beta - \frac\beta\alpha}+1) & \leq \frac{|p|^\alpha}{\theta^{\alpha-1}\alpha}  + 
   \frac{\theta}{\alpha'}(m^{\beta - \frac\beta\alpha}+1)^{\alpha'} \leq \frac{|p|^\alpha}{\theta^{\alpha-1}\alpha} + \theta \tilde{c} (m^{\beta}+1),
   %\\|p| & \leq \frac{|p|^\alpha}{\delta\alpha}  + \frac{\delta^{\frac{\alpha'}{\alpha}}}{\alpha'}.
   \end{aligned}
\end{equation}
where $\tilde{c}$ depends only on $\alpha$ and $\beta$. Hence, using \eqref{eq:Young.simple.case} with $\theta=\frac{1}{2C^2\tilde{c}}$ combined with \eqref{eq:case1-lb1}, and redefining the constant $C$ appropriately, we conclude the proof of \eqref{eq:boundsH}.

Finally, to prove \eqref{eq:boundsmDpH}, we start by noting that   Remark~\ref{rmk:implied-a} and the notation in \eqref{a:mDpHat0} allow us to write  $mD_pH(x,Du,m) = 0$ in the set $\{x\colon m=0\}$ even if $D_pH(x,Du,m)$ is not defined. Then, we  observe that \eqref{eq:gamma} yields
\begin{equation*}
    \frac{1}{\bar\gamma'} = \frac{1}{\bar\beta} + \frac{\alpha-1}{\bar\gamma} \quad \text{and} \quad \frac{1}{\bar\gamma'} = \frac{1}{\bar\beta} \Big( \beta + 1 -\frac{\beta}{\alpha}\Big).
\end{equation*}
Hence, to conclude \eqref{eq:boundsmDpH}, it suffices to multiply the estimate in \eqref{eq:assH.DpH.upper.simple} by $m$ and use H\"older's inequality.
\end{proof}

Next, we provide two pairs of conditions that constitute an equivalent alternative to \eqref{eq:assH.lower.simple}.

\begin{lemma}[On condition \eqref{eq:assH.lower.simple}]\label{lem:lower.simple}
Assume that $H\colon \Tt^d\times \RR^d \times \RR^+ \to \RR$ satisfies Assumption~\ref{onH.monotone}, \eqref{eq:assH.upper.simple}, and \eqref{eq:assH.DpH.upper.simple}. Then, each of the following two pairs of conditions is an equivalent alternative to \eqref{eq:assH.lower.simple}:
\begin{equation}
\begin{cases}
       H(x,0,m) \geq -C(  m^{\beta}+1),\\
       D_p H(x,p,m) \cdot p \geq \frac{1}{C}|p|^\alpha - C(m^\beta + 1), 
    \end{cases} \label{eq:alternative1-simple} 
\end{equation}    
and
\begin{equation}
     \begin{cases}
         H(x,p,m) \geq -C(  m^{\beta}+1),\\
        D_p H(x,p,m) \cdot p - H(x,p,m) \geq \frac{1}{C}(|p|^\alpha + m^\beta) - C.
    \end{cases} \label{eq:alternative2-simple}
\end{equation}
\end{lemma}

\begin{proof}
We start by observing that Assumption~\ref{onH.monotone} and~\eqref{eq:assH.upper.simple} together imply \eqref{eq:DpHdotpminusH-estimate1} by Lemma~\ref{lem:A1+anyother}. Then, \eqref{eq:assH.lower.simple} together with \eqref{eq:DpHdotpminusH-estimate1} implies \eqref{eq:alternative1-simple}. On the other hand, \eqref{eq:alternative2-simple} implies \eqref{eq:alternative1-simple} by simple arithmetic. We now show that \eqref{eq:alternative1-simple} together with the given assumptions implies the other two.

First, \eqref{eq:alternative1-simple} and \eqref{eq:assH.DpH.upper.simple} imply \eqref{eq:assH.lower.simple}. In fact, let $\delta>1$. Using first the fundamental theorem of calculus, and then \eqref{eq:alternative1-simple} to estimate the first two terms on the resulting right-hand side  and \eqref{eq:assH.DpH.upper.simple} for the third term,  we obtain  the following estimate: 
    \begin{equation*}
        \begin{aligned}
            H(x,p,m) &= H(x,0,m)  + \int_{1/\delta}^1 D_p H(x,pt,m) \cdot p \,dt + \int_0^{1/\delta} D_p H(x,pt,m) \cdot p \,dt \\
            & \geq \left(\frac{\delta^\alpha-1}{C\alpha\delta^\alpha} -C\frac{1}{\alpha\delta^\alpha}\right) |p|^\alpha - \frac{C}{\delta}  |p|\Big(m^{\beta - \frac\beta\alpha} +1\Big) -C(m^\beta +1)(\ln{\delta} +1)
        \end{aligned}
    \end{equation*}
for a.e.~$x\in\Tt^d$ and for all $(p,m)\in\RR^d\times\RR^+$. In particular, choosing $\delta = (2C^2+1)^{1/\alpha}$ and redefining the constant $C$ appropriately, we get
\begin{equation}\label{eq:case1-lb1.ext}
    H(x,p,m)\geq \frac{1}{C}|p|^\alpha - C|p|\Big(m^{\beta - \frac\beta\alpha} +1\Big) -C(m^\beta +1)
\end{equation}
for a.e.~$x\in\Tt^d$ and for all $(p,m)\in\RR^d\times\RR^+$. Finally, we obtain \eqref{eq:assH.lower.simple} by combining \eqref{eq:case1-lb1.ext} and \eqref{eq:Young.simple.case} with $\theta=(2C^2/\alpha)^{1/(\alpha-1)}$. 

Secondly, \eqref{eq:alternative1-simple} together with Assumption~\ref{onH.monotone}, \eqref{eq:assH.upper.simple}, and \eqref{eq:assH.DpH.upper.simple} implies \eqref{eq:alternative2-simple}. It is enough to derive
\begin{equation}\label{eq:DpH-H}
        D_pH(x,p,m)\cdot p - H(x,p,m) \geq \frac{1}{C} ( |p|^{\alpha} +m^{\beta}) - C,
\end{equation}
because the other estimate is implied by \eqref{eq:assH.lower.simple}, which is already explained in the previous paragraph. In fact, using the fundamental theorem of calculus, we observe that \begin{equation}\label{eq:DpHdotpminusH-integral-formula}\begin{aligned}
       &D_pH(x,p,m)\cdot p - H(x,p,m) \\ 
       &\qquad = -H(x,0,m) + \int_0^1 \big(D_pH(x,p,m)-D_pH(x,pt,m)\big)\cdot p\,dt. 
    \end{aligned}
    \end{equation}
By Remark~\ref{rmk:cxty}, the map $p\mapsto H(x,p,m)$ is convex at any fixed $(x,m)$; hence, the integrand in \eqref{eq:DpHdotpminusH-integral-formula} is non-negative, so we do not increase the integral when we reduce its domain of integration. Thus, using \eqref{eq:assH.upper.simple}, \eqref{eq:alternative1-simple}, and \eqref{eq:assH.DpH.upper.simple}, we get for any $\delta>1$ that
\begin{equation}\label{eq:DpHdotpminusH-estimate2}
\begin{aligned}
       &D_pH(x,p,m)\cdot p - H(x,p,m) \\ 
       &\qquad \geq -H(x,0,m) + \int_0^{1/\delta} (D_pH(x,p,m) -D_pH(x,pt,m))\cdot p \,dt  \\
       &\qquad \geq -H(x,0,m) + \frac{1}{\delta}D_pH(x,p,m)\cdot p - |p|\int_0^{1/\delta} |D_pH(x,pt,m)|\,dt \\
       &\qquad \geq \left(\frac{1}{C\delta} -C\frac{1}{\alpha\delta^\alpha}\right)|p|^\alpha -C\frac{1}{\delta}(m^{\beta-\frac{\beta}\alpha}+1)|p| + \left(\frac{1}{C}-C\frac{1}{\delta}\right)m^\beta - C\left(1+\frac{1}{\delta}\right).
    \end{aligned}
\end{equation}
Then,  choosing $\delta \geq \max\{2C^2, (2C^2/\alpha)^{1/(\alpha-1)},1\}$, \eqref{eq:DpHdotpminusH-estimate2} yields
\begin{equation}\label{eq:DpH-H-new}
    D_pH(x,p,m)\cdot p - H(x,p,m) \geq \frac{1}{\delta}\left(\frac{1}{2C}|p|^\alpha + \frac{\delta}{2C}m^\beta -C(m^{\beta-\frac{\beta}\alpha}+1)|p|\right) - 2C.
\end{equation}
%Now, using \eqref{eq:Young.simple.case} with $\theta = \big(\frac{4C^2}{\alpha}\big)^{\frac{1}{\alpha-1}}$ to estimate the %right-hand side of \eqref{eq:DpH-H-new}, and then choosing
% $\delta$  sufficiently large, the negative term inside the parantheses on the right-hand side of \eqref{eq:DpH-H-new} can %be absorbed into the two positive terms therein, from which we conclude \eqref{eq:DpH-H} by redefining $C$.
{ On the other hand, it follows from  \eqref{eq:Young.simple.case} with $\theta = \big(\frac{4C^2}{\alpha}\big)^{\frac{1}{\alpha-1}}$ that
 \begin{equation*}
 -C(m^{\beta-\frac{\beta}\alpha}+1)|p| \geq -  \frac{1}{4C}|p|^\alpha - \hat{c}(m^\beta+1),
 \end{equation*}
where  $\hat{c}>0$ depends only on $\alpha$, $\beta$, and $C$. Substituting this back into \eqref{eq:DpH-H-new} yields
\[
    D_pH(x,p,m)\cdot p - H(x,p,m) \geq \frac{1}{\delta}\left(\frac{1}{4C}|p|^\alpha + \left(\frac{\delta}{2C} - \hat{c}\right)m^\beta - \hat{c}\right) - 2C.
\]
Finally, we choose $\delta$ sufficiently large so that $\frac{\delta}{2C} - \hat{c} \geq \frac{1}{4C}$, from which we conclude \eqref{eq:DpH-H} by redefining $C$ therein.
  } 
\end{proof}

We proceed by showing similar results to those in Lemma~\ref{lem:ADFGU} and Lemma~\ref{lem:lower.simple} within the framework of Assumption~\ref{onH.congestiongrowth}.

\begin{lemma}[{cf.~\cite[Lemma~2.3 and~2.4]{ADFGU}}]\label{lem:ADFGU.congestion}
Assume that $H:\Tt^d\times \RR^d \times \RR^+\to \RR$ satisfies Assumption~\ref{onH.congestiongrowth}. Then, there exists a constant, $C>0$, for which the following bound holds for a.e.~$x\in\Tt^d$ and for all $(p,m)\in\RR^d\times\RR^+$:
\begin{equation}\label{eq:boundsH+}
        \begin{aligned}
                 H(x,p,m) \leq C\left(\frac{1}{m} + \frac{1}{m^\tau}\right)(|p|^{\alpha(1+\frac{\tau}{\beta})} +1) - \frac{1}{C}m^{\beta}.
        \end{aligned}
\end{equation}
Moreover, evaluating $mD_pH(x, Du, m)$  in the set $\{x\colon m = 0\}$ in the sense of \eqref{a:mDpHat0}, we have for the parameters $\bar\beta>1$ and $\bar\gamma>1$ defined in \eqref{eq:gamma} that
\begin{equation}\label{eq:boundsmDpH+}
        \begin{aligned}
                mD_pH(\cdot, Du, m) \in L^{\bar\gamma'}(\Tt^d;\RR^d)\enspace \hbox{ whenever } \enspace m\in L^{\bar \beta}(\Tt^d;\RR^+_0) \hbox{ and } u\in W^{1,\bar\gamma}(\Tt^d).
        \end{aligned}
\end{equation}
\end{lemma}

\begin{proof} 
        The arguments to prove this lemma are similar to those used to prove Lemma~\ref{lem:ADFGU}, for which reason we only highlight the main steps.

    Using the fundamental theorem of calculus as in Lemma~\ref{lem:ADFGU}, but using \eqref{eq:assH.upper.congestion} and \eqref{eq:assH.DpH.upper.congestion} in place of \eqref{eq:assH.upper.simple} and \eqref{eq:assH.DpH.upper.simple}, we conclude that
\begin{equation}\label{eq:boundsH+1}
    \begin{aligned}
        & H(x,p,m) \leq - \frac{1}{C}m^{\beta} + C +C \left(\frac{1}{m} + \frac{1}{m^\tau}\right)|p|^{{\alpha(1+\frac{\tau}{\beta})}} + C \left( m^{\beta-\frac{\beta}\alpha} + \frac{1}{m} \right)|p|.
    \end{aligned}
\end{equation}
{To estimate the last term in \eqref{eq:boundsH+1}, we first split it using the bound $|p| \leq |p|^{\alpha(1+\frac{\tau}{\beta})} + 1$ to obtain the intermediate estimate
\begin{equation}\label{eq:intermediate-step}
    C \left( m^{\beta-\frac{\beta}\alpha} + \frac{1}{m} \right)|p| \leq C \big( m^{\beta-\frac{\beta}\alpha} + 1 \big)|p| + \frac{C}{m}\big(|p|^{\alpha(1+\frac{\tau}{\beta})} + 1\big).
\end{equation}
Next, for the first term on the right-hand side of \eqref{eq:intermediate-step},}
in a similar way to \eqref{eq:Young.simple.case}, we observe that Young's inequality yields for every $\theta>0$ that
\begin{equation}\label{eq:Young.congestion.case}
    \begin{aligned}
           &|p|(m^{\beta-\frac{\beta}\alpha}+1)\\ & \quad \leq \frac{|p|^{\alpha(1+\frac{\tau}{\beta})}}{\theta^{\alpha(1+\frac{\tau}{\beta})-1} \alpha(1+\frac{\tau}{\beta}) (m^\tau +1)}  + \theta\frac{\alpha(1+\frac{\tau}{\beta})-1}{\alpha(1+\frac{\tau}{\beta})}(m^\tau +1)^{\frac{1}{\alpha(1+\frac{\tau}{\beta})-1}} \big( m^{\beta-\frac{\beta}\alpha}+1\big)^{\frac{\alpha(1+\frac{\tau}{\beta})}{\alpha(1+\frac{\tau}{\beta})-1}}\\
         &\quad\leq \tilde{c}\left(\frac{|p|^{\alpha(1+\frac{\tau}{\beta})}}{\theta^{\alpha(1+\frac{\tau}{\beta})-1}(m^\tau +1)} + \theta (m^\beta+1)\right),
    \end{aligned}
\end{equation}
where $\tilde{c}$ depends only on $\alpha$, $\beta$, and $\tau$, and where we used the identity
\begin{equation*}%\label{eq:bound.beta.tau}
    \frac{\tau}{\alpha(1+\frac{\tau}{\beta})-1}+ \left(\beta - \frac{\beta}{\alpha}\right)\left(\frac{\alpha(1+\frac{\tau}{\beta})}{\alpha(1+\frac{\tau}{\beta})-1}\right)=\beta.
\end{equation*}
{In addition, we have the Young-type inequality
\begin{equation}\label{eq:Young-like.congestion}
    \begin{aligned}
        1 \leq \frac{1}{\theta^{\frac{1}{\beta}}m} + \theta m^\beta.
    \end{aligned}
\end{equation}
Thus, \eqref{eq:boundsH+} follows from combining 
\eqref{eq:boundsH+1} with \eqref{eq:intermediate-step}, \eqref{eq:Young.congestion.case}, and \eqref{eq:Young-like.congestion} using an appropriately small value of $\theta$.}

We are left to prove \eqref{eq:boundsmDpH+}. From \eqref{eq:gamma}, we get that
\begin{equation*}
    \frac{1}{\bar\gamma'} = \frac{1-\tau}{\bar\beta} + \frac{{\alpha(1+\frac{\tau}{\beta})}-1}{\bar\gamma} \quad \text{and} \quad \frac{1}{\bar\gamma'} = \frac{1}{\bar\beta} \Big( \beta + 1 -\frac{\beta}{\alpha}\Big).
\end{equation*}
Thus, multiplying  \eqref{eq:assH.DpH.upper.congestion} by $m$ while noting that it is valid even when $m=0$ by Remark~\ref{rmk:implied-a}, and using H\"older's inequality, we obtain \eqref{eq:boundsmDpH+}.
\end{proof}

\begin{remark}[On condition \eqref{eq:assH.lower.congestion}]\label{rmk:lower.congestion}
    Assume that  $H\colon \Tt^d\times \RR^d \times \RR^+ \to \RR$ satisfies the following bounds for a.e.~$x\in\Tt^d$ and for all $(p,m)\in\RR^d\times\RR^+$: 
\begin{equation}\label{eq:alternative-congestion}
    \begin{aligned}
       &  H(x,0,m) \geq -C(  m^{\beta}+1),\\
        & D_p H(x,p,m) \cdot p \geq \frac{1}{C}\left(\frac{1}{m^\tau+1}\right)|p|^{\alpha(1+\frac{\tau}{\beta})} - C\Big(m^{\beta - \frac{\beta}{\alpha}} + 1\Big)|p|.    
    \end{aligned}
\end{equation}
Then, as we show next, \eqref{eq:assH.lower.congestion} holds. As such, \eqref{eq:alternative-congestion} can be considered an alternative to \eqref{eq:assH.lower.congestion} in the framework of Assumption~\ref{onH.congestiongrowth}; however, \eqref{eq:alternative-congestion} is strictly stronger since it is not implied by Assumption~\ref{onH.monotone} and Assumption~\ref{onH.congestiongrowth} combined.

To show the aforementioned implication, we apply the fundamental theorem of calculus as in Lemma~\ref{lem:ADFGU},  but using \eqref{eq:alternative-congestion} in place of \eqref{eq:assH.upper.simple} and \eqref{eq:assH.DpH.upper.simple}, to get that
\begin{equation}\label{eq:boundsH+1bis}
    \begin{aligned}
        & H(x,p,m) \geq - C(m^\beta + 1) + \frac{1}{{\alpha(1+\frac{\tau}{\beta})} C} \left(\frac{1}{m^\tau + 1}\right)|p|^{\alpha(1+\frac{\tau}{\beta})} -C \big(  m^{\beta-\frac{\beta}\alpha}+1\big) |p|.
    \end{aligned}
\end{equation}
Hence, choosing $\theta$ in \eqref{eq:Young.congestion.case} appropriately, as in the preceding proof, \eqref{eq:assH.lower.congestion} follows from 
\eqref{eq:boundsH+1bis} and \eqref{eq:Young.congestion.case}.
\end{remark}

    %     D_pH(x,p,m)\cdot p - H(x,p,m) \geq C^{-1}|p|^{\alpha} - C(m^{\beta}+1),
    % \end{equation}
    % for some positive constant  $C$, then we can do a suitable convex combination between the preceding inequality and
    % \eqref{eq:DpHdotpminusH-estimate1}     
    % to  conclude \eqref{eq:DpH-H},  with a possibly different constant $C$.
    % Moreover, it is enough to show \eqref{eq:DpHdotpminusH-estimate2} assuming that $|p| \geq c(m^{\frac{\beta}{\alpha}}+1)$ for some positive constant $c$, since it otherwise holds anyway by \eqref{eq:DpHdotpminusH-estimate1}, perhaps after redefining $C$. Specifically, we assume that
    % \[|p|\geq (9C^2(m^{\beta-\frac{\beta}{\alpha}}+1))^{1/(\alpha-1)}\]
    % where $C$ is the constant in \eqref{eq:assH.DpH.upper.simple}--\eqref{eq:assH.upper.simple}, and prove \eqref{eq:DpHdotpminusH-estimate2} for some possibly different positive constant $C$, which we do not relabel.
    % For this, we observe that \eqref{eq:assH.DpH.upper.simple}--\eqref{eq:assH.lower.simple} implies the estimate
    % \[D_pH(x,pt,m)\cdot p \leq \frac{1}{2} D_pH(x,p,m)\cdot p,\]
    % as long as $0\leq t\leq (3C^2)^{-1/(\alpha-1)}$. Consequently, \eqref{eq:DpHdotpminusH-integral-formula} implies that
    % \[D_pH(x,p,m)\cdot p - H(x,p,m) \geq -H(x,0,m) + \frac{1}{2(3C^2)^{1/(\alpha-1)}}D_pH(x,p,m)\cdot p.\]
    % Using \eqref{eq:assH.lower.simple}--\eqref{eq:assH.upper.simple} and redefining $C$, the last estimate yields
    % \[D_pH(x,p,m)\cdot p - H(x,p,m) \geq \frac{1}{C}|p|^{\alpha} - C(m^{\beta-\frac{\beta}\alpha}+1)|p|.\]
    % Then, we conclude \eqref{eq:DpHdotpminusH-estimate2} by Young's inequality.

\subsection{Monotone Operator Theory}\label{subs:monotone}

The proofs of existence of solutions that we present in the subsequent sections rely on an  abstract existence theorem for monotone operators on Banach spaces, Theorem~\ref{thm:monotone.abstract-weak} below. Before stating this theorem, we recall the notion of monotonicity, coercivity, and hemicontinuity used therein (and  Remark~\ref{rmk:hemicontinuous-upgrade} below). These properties are standard in the theory of variational inequalities.

\begin{definition}\label{def:prelim}
    Let \(X\) be a reflexive Banach space, and let \(X'\) denote its dual. Let    \(A\colon D(A)\subset X \to X' \) be an operator on $X$, with convex domain $D(A)$. We say that:
    \begin{enumerate}[label=(\alph*)]
        
    \item $A$ is monotone if we have for all $v_1, v_2\in D(A)$  that \[\langle A[v_1] - A[v_2], v_1 - v_2 \rangle_{X', X} \geq 0.\]

    \item\label{defitem:coercive} $A$ is coercive if there exists $\bar v \in D(A)$ such that
\[\lim_{v\in D(A),\, \Vert v\Vert_X \to \infty} \frac{\langle A[v] - A[\bar v], v - \bar v \rangle_{X', X}}{\Vert v\Vert_X} = +\infty.\]
     
     \item\label{defitem:hemicont} $A$ is hemicontinuous if  we have for all $u, v\in D(A)$ and $w\in X$  that  the real-valued map \[t\mapsto \langle A[(1-t)u + tv], w \rangle_{X', X}\]
     is continuous in the interval $[0,1]$.
    \end{enumerate}
\end{definition}

The cornerstone of our analysis for the congestion and weak-growth cases (Theorems~\ref{thm:main-congestion.growth} and \ref{thm:main-weak.power.growth}) is the following existence theorem. It guarantees solutions to weak variational inequalities under minimal assumptions by extending the classical Debrunner--Flor theorem to unbounded domains.

\begin{theorem}\label{thm:monotone.abstract-weak}
Let \(X\) be a reflexive Banach space, and let \(X'\) denote its dual. Let \(\mathcal{K} \subset X\)  be a nonempty convex subset of \(X\), and assume that the operator \(A\colon \mathcal{K} \to X' \) is monotone and coercive. Then, 
\begin{equation}\label{eq:weakIneqA}
\begin{aligned}
\exists \, u\in  \overline{\mathcal{K}}\colon \quad \forall\, v\in  \mathcal{K}, \enspace \langle A[v], v-u\rangle_{X',X} \geq 0.
\end{aligned}
\end{equation}
\end{theorem}

\begin{proof}
    The case of bounded $\mathcal{K}$ is an immediate corollary of the Debrunner--Flor theorem \cite{debrunner-flor}. In this proof, we use this result where $\mathcal{K}$ is replaced by some $C\subset\mathcal{K}$, which is the convex hull of a finite set of points; i.e.,~a finite-dimensional, closed, convex polytope. Moreover, we suppose, without loss of generality, that $0\in\mathcal{K}$ and $A$ is coercive around the point $0\in\mathcal{K}$.

    For $v\in\mathcal{K}$, define
    \[\mathcal{S}(v) := \{u\in X\colon \langle A[v], v-u\rangle_{X',X} \geq 0\}.\]
    To prove \eqref{eq:weakIneqA}, it suffices to show that
    \[\overline{\mathcal{K}} \cap \bigcap_{v\in\mathcal{K}} \mathcal{S}(v) \neq \emptyset.\]
    In fact, we will show that there exists a closed ball, $\bar B_r(0)\subset X$, centered at the origin with radius $r>0$,  such that
    \[\bar B_r(0)\cap \overline{\mathcal{K}} \cap \bigcap_{v\in\mathcal{K}} \mathcal{S}(v) \neq \emptyset.\]

    In the weak topology of $X$, the set $\bar B_r(0)\cap \overline{\mathcal{K}}$ is compact (Banach--Alaoglu theorem) and the sets $\mathcal{S}(v)$ are closed; hence,  by the finite intersection property, the preceding condition follows once we show that \begin{equation}\label{exists.bounded.galerkin}
        B_r(0)\cap \overline{\mathcal{K}} \cap \mathcal{S}(v_1) \cap \ldots \cap \mathcal{S}(v_n) \neq \emptyset
    \end{equation}
    for any finite collection $v_1,\ldots, v_n\in\mathcal{K}$. Let $C\subset\mathcal{K}$
    %\[C := \operatorname{conv}\{0, v_1,\ldots,v_n\}\subset \mathcal{K}\]
    denote the convex hull of $\{0, v_1,\ldots,v_n\}$. As we remarked in the beginning, we have
    \[C\cap \bigcap_{v\in C} \mathcal{S}(v) \neq \emptyset.\]
Let $u\in C\cap \bigcap_{v\in C} \mathcal{S}(v)$. Then, we have $u\in S(u/2)$ since $u\in C$ implies $u/2\in C$.
    To conclude \eqref{exists.bounded.galerkin} from this, we simply observe that $u\in \mathcal{S}(u/2)$ implies $\|u\|_X \leq r$ for some uniformly fixed value of $r$ because of the coercivity of $A$.
\end{proof}

\begin{remark}\label{rmk:hemicontinuous-upgrade}
    In the setting of the previous theorem, assume further that  $u\in\mathcal{K}$ and that $A$ is hemicontinuous. Then, using \eqref{eq:weakIneqA} with $v$ replaced by $(1-t)u+tv$ for $0<t\leq 1$, and subsequently dividing the resulting inequality by $t$ and letting $t$ approach 0, we  conclude that $u$ satisfies the so-called (strong) variational inequality associated with a monotone operator:
\begin{equation*}
\begin{aligned}
\forall\, v\in  \mathcal{K}, \enspace \langle A[u], v-u\rangle_{X',X} \geq 0.
\end{aligned}
\end{equation*}

We further remark that a similar result on the existence of solutions to the strong variational inequality has been established in \cite[Thm.~1.7]{KiSt00}, under the stronger assumption on $A$ of continuity on finite-dimensional subspaces in place of hemicontinuity.
\end{remark}

\subsection{Discussion on the Solution Concepts}\label{subs:solutions}

The concept of weak solutions, as presented in Definition \ref{def:weak}, has its foundations in the weak variational inequality framework developed for monotone operators. Defining an operator $A$ by setting
\begin{equation*}%\label{defA}
        \begin{aligned}
                A \begin{bmatrix}\mu \\
                        \upsilon \\
                \end{bmatrix}
                := \begin{bmatrix}\displaystyle -\upsilon-  H(x,D\upsilon,\mu) + V(x)  \\
                        \mu -  \div(\mu D_pH(x, D\upsilon,\mu)) -1  \\
                \end{bmatrix},
        \end{aligned}
\end{equation*}
whose domain $D(A)$ contains at least all pairs $(\mu,\upsilon)\in L^{\infty}(\Tt^d)\times W^{1,\infty}(\Tt^d)$ with $\essinf\mu > 0$, one can reformulate \eqref{def:weak-solution-concept} as
\begin{equation*}%\label{defvws}
        \begin{aligned}
                \left\langle  \begin{bmatrix}\mu \\
                        \upsilon \\
                \end{bmatrix}
                -  \begin{bmatrix}m \\
                        u \\
                \end{bmatrix}, A \begin{bmatrix}\mu \\
                        \upsilon \\
                \end{bmatrix} \right\rangle_{L^1(\Tt^d) \times
                        W^{1,1}(\Tt^d), L^\infty(\Tt^d) \times
                        W^{-1,\infty}(\Tt^d)} \geq 0.
        \end{aligned}
\end{equation*}
Note that our (standing) Assumption~\ref{onH.monotone} ensures that $A$ is a monotone operator. 

On the other hand, Definition~\ref{def:strong} provides a stronger notion of solution. In our proofs of Theorems~\ref{thm:main-simple.power.growth}~and~\ref{thm:main-congestion.growth}, we first establish the existence of a weak solution $(m,u)$ in the sense of Definition~\ref{def:weak}, albeit with more integrability on $(m,u)$, and such that \eqref{def:weak-solution-concept} holds for $(\mu,\upsilon)$ with less integrability. Then, we prove that $(m,u)$ satisfies \eqref{def:hjb-strong} and \eqref{def:transport-strong} by judicious choices of test functions $(\mu,\upsilon)$ in \eqref{def:weak-solution-concept}. It is also noteworthy that Definition~\ref{def:weak} already allows for less regular test functions than  earlier papers addressing existence of weak solutions obtained by monotonicity methods \cite{ FG2,FGT1,FeGoTa21}; this reduced regularity gap between the solution and the test functions brings the notion of weak solutions in this manuscript closer to that of strong solutions.

We now justify that a solution in the sense of Definition~\ref{def:strong} is also a solution in the sense of Definition~\ref{def:weak}, provided that there exists $q>1$ for which
\[m\in L^{q'}(\Tt^d;\RR^+_0),\quad mD_pH(x,Du,m)\in L^{q'}(\Tt^d;\RR^d), \quad u\in W^{1,q}(\Tt^d).\]

Indeed, Assumption~\ref{onH.monotone} yields
\[\begin{aligned}
&\big(- H(x,D\upsilon,\mu) +  H(x,Du,m)\big) (\mu - m)\\ &\quad + 
\big( \mu D_p H(x,D\upsilon,\mu) - m D_p H(x,Du,m)\big) \cdot (D\upsilon-Du) \geq 0\end{aligned}\]
a.e.~on $\Tt^d$. On the other hand, \eqref{def:hjb-strong} implies
\[H(x,Du,m)(\mu-m) \leq -(u-V(x))(\mu-m)\]
a.e.~on $\Tt^d$. Combining and rearranging, we get
\[\begin{aligned}
& \big(  -\upsilon- H(x, D\upsilon,\mu) + V(x)  \big)(\mu
\, -m) + \mu D_p H(x, D\upsilon,\mu) \cdot(D\upsilon - Du)  + (\mu-1)( \upsilon-u) \\
& \qquad \geq (m-1)(\upsilon-u) + mD_pH(x,Du,m)\cdot (D\upsilon-Du)
\end{aligned}\]
a.e.~on $\Tt^d$. Now, integrating both sides and using \eqref{def:transport-strong} with $\varphi = \upsilon-u$, and taking into account that the right-hand side is integrable, we conclude that \eqref{def:weak-solution-concept} holds.

\section{Proof of Theorem~\ref{thm:main-simple.power.growth}} 
\label{sec:proof-power-growth}

Here, we establish our first main result: under the power-growth conditions of Assumption~\ref{onH.powergrowth}, the stationary MFG system \eqref{smfg} admits a strong solution. 
Let $\mathcal K
=
L^{\bar\beta}(\Tt^d;\RR^+_0)\times W^{1,\bar\gamma}(\Tt^d)$.
For each $\varepsilon>0$, we introduce a coercive Banach-space operator
\[
A_\varepsilon
\;:\;
\mathcal K
\;\to\;
L^{\bar\beta'}(\Tt^d)\times W^{-1,\bar\gamma'}(\Tt^d),
\]
defined in \eqref{defAepsi} by adding a $\bar\gamma$-Laplacian regularization to the original MFG operator. We then solve the associated variational inequality via Theorem~\ref{thm:monotone.abstract-weak}, derive uniform a priori estimates independent of~$\varepsilon$, and pass to the limit $\varepsilon\to0$ using compactness and Minty's method, thereby recovering a strong solution of \eqref{smfg}. In contrast to earlier Hilbert-space techniques, which rely on higher-order smoothing and yield only weak solutions, our Banach-space approach produces strong solutions. In subsequent sections, we adapt the same strategy, suitably modified, to handle the singularities arising in congestion models as well as weaker growth conditions. 

\medskip

Based on the hypotheses of Theorem~\ref{thm:main-simple.power.growth}, Remarks~\ref{rmk:cxty} and~\ref{rmk:implied-a}, Lemma~\ref{lem:ADFGU}, and the fact that $V\in L^{\bar\beta'}(\Tt^d)$, we define the regularized operator
\[
A_\varepsilon
\;:\;
\mathcal K
\;\to\;
L^{\bar\beta'}(\Tt^d)\times W^{-1,\bar\gamma'}(\Tt^d),
\]
by setting, for all $(m,u)\in \mathcal K$ and $(\eta,\nu)\in L^{\bar\beta}(\Tt^d)\times W^{1,\bar\gamma}(\Tt^d)$,
\begin{equation}\label{defAepsi}
\begin{aligned}
\big\langle A_\varepsilon[m,u],(\eta,\nu)\big\rangle
&=\int_{\Tt^d}\bigl(-u -H(x,Du,m)+V(x)\bigr)\,\eta\,dx\\
&\quad+\int_{\Tt^d}\bigl(m\,D_pH(x,Du,m)\cdot D\nu + (m-1)\,\nu\bigr)\,dx\\
&\quad+\varepsilon\int_{\Tt^d}\bigl(|Du|^{\bar\gamma-2}Du\cdot D\nu
+|u|^{\bar\gamma-2}u\,\nu\bigr)\,dx.
\end{aligned}
\end{equation}
The last term on the right-hand side of \eqref{defAepsi} is the G\^ateaux derivative of the map $u\mapsto \frac{1}{\bar \gamma}\|u\|_{W^{1,\bar\gamma}}^{\bar\gamma}$ and ensures that $A_\varepsilon$ is coercive for $\varepsilon>0$. We begin by verifying this coercivity (Definition~\ref{def:prelim}\ref{defitem:coercive}).

\begin{proposition}[Coercivity of \(A_\epsi\)]\label{prop:coercAepsi} 
Under the assumptions of Theorem~\ref{thm:main-simple.power.growth},
the operator $A_\epsi$ in \eqref{defAepsi} is coercive for any $\epsi>0$.
\end{proposition}

\begin{proof}
For any $(m,u) \in \mathcal K$, we have
\begin{equation}\label{eq:coercive-test.Aepsi}
                \begin{aligned}
                        &\langle A_\epsi[m,u]
                        - A_\epsi[1,0] , (m-1,u)  \rangle\\
            &\quad = \int_{\Tt^d} \Big[ m\big( -H(x, Du,m) + D_p H(x, Du,m) \cdot Du\big) + H(x, Du,m) + \epsi |D u|^{\bar\gamma} + \epsi |u|^{\bar\gamma } \\
                        &\qquad  + \, m H(x, 0,1) - H(x, 0,1) - D_p H(x, 0, 1) \cdot Du \Big]\, dx.
                \end{aligned}
        \end{equation}
We estimate the first term on the right-hand side of the preceding
identity by \eqref{eq:DpHdotpminusH-estimate1} of Lemma~\ref{lem:A1+anyother}, the second term by \eqref{eq:assH.lower.simple}, and the terms
in the last line by noting that $|H(x,0,1)| \leq C$ and $|D_pH(x,0,1)| \leq C$  by \eqref{eq:assH.upper.simple}-\eqref{eq:assH.DpH.upper.simple}-\eqref{eq:assH.lower.simple}, while keeping Remark~\ref{rmk:implied-a} in mind. Combining these estimates with Young's inequality, we obtain
 \begin{equation}\label{eq:coercive.Aepsi}
                \begin{aligned}
                        &\langle A_\epsi[m,u]
                        - A_\epsi[1,0] , (m-1,u)  \rangle\\
                        &\quad\geq \int_{\Tt^d}\bigg(  \frac{1}{C} \Big(
            m^{\beta+1} +  |Du|^\alpha\Big)    
            + \epsi |Du|^{\bar\gamma} + 
                        \epsi | u|^{\bar\gamma}  - C (m + m^\beta + |Du|) - C\bigg)\,
                        dx \\ 
                        &\quad\geq \int_{\Tt^d}\bigg(  \frac{1}{C} \Big(
            m^{\beta+1} +  |Du|^\alpha\Big)  
            + \epsi |Du|^{\bar\gamma} + 
                        \epsi | u|^{\bar\gamma}  - C  \bigg)\,
                        dx \\
                        &\quad\geq \frac{1}{C} \Vert m\Vert_{L^{\bar\beta}(\Tt^d)}^{\bar\beta} + \epsi \Vert u\Vert_{W^{1,\bar\gamma}(\Tt^d)}^{\bar\gamma}
            - C.  
                \end{aligned}
        \end{equation} 
%where we also used the convention on constants  in Remark~\ref{rmk:onC}.
From the preceding estimate, we conclude that 
\begin{equation*}%\label{eq:coercAepsi}
                \begin{aligned}
                        \lim_{\Vert (m,u)\Vert_{L^{\bar\beta}(\Tt^d) \times W^{1,\bar\gamma}(\Tt^d)} \to+\infty \atop
                                (m,u) \in L^{\bar\beta}(\Tt^d;\RR^+_0) \times W^{1,\bar\gamma}(\Tt^d) }\frac{\langle A_\epsi[m,u] - A_\epsi[1, 0] , (m-1,u)  \rangle}{\Vert (m,u)\Vert_{L^{\bar\beta}(\Tt^d) \times W^{1,\bar\gamma}(\Tt^d)}} =+ \infty
                \end{aligned}
        \end{equation*}
for any $\epsi>0$, which shows that $A_\epsi$ is coercive.
\end{proof}

\begin{lemma}[Hemicontinuity of $A_\varepsilon$]\label{lem:hemicont}
Under the assumptions of Theorem~\ref{thm:main-simple.power.growth}, the map
\[
A_\varepsilon
\;:\;
\mathcal K
\;\longrightarrow\;
L^{\bar\beta'}(\Tt^d)\times W^{-1,\bar\gamma'}(\Tt^d)
\]
defined in \eqref{defAepsi} is hemicontinuous in the sense of Definition~\ref{def:prelim}\ref{defitem:hemicont}.  That is, for any
\[
(\mu,\upsilon),\;(\bar\mu,\bar\upsilon)\in \mathcal K
\quad\text{and}\quad
(\eta,\nu)\in L^{\bar\beta}(\Tt^d)\times W^{1,\bar\gamma}(\Tt^d),
\]
the map
\[
t\;\longmapsto\;
\bigl\langle A_\varepsilon\bigl[(1-t)(\bar\mu,\bar\upsilon)+t(\mu,\upsilon)\bigr],(\eta,\nu)\bigr\rangle
\]
is continuous on $[0,1]$.
\end{lemma}

\begin{proof}
Let $t\in[0,1]$, set 
\[
(\mu_t,\upsilon_t)
:=
(1-t)(\bar\mu,\bar\upsilon)+t(\mu,\upsilon),
\]
and fix a test pair $(\eta,\nu)\in L^{\bar\beta}(\Tt^d)\times W^{1,\bar\gamma}(\Tt^d)$.  By \eqref{defAepsi},
\[
\begin{aligned}
\bigl\langle A_\varepsilon[\mu_t,\upsilon_t],(\eta,\nu)\bigr\rangle
&=
\int_{\Tt^d}
\bigl(-\upsilon_t-H\bigl(x,D\upsilon_t,\mu_t\bigr)+V(x)\bigr)\,\eta\,dx
\\
&\quad
+\int_{\Tt^d}\bigl(\mu_t\,D_pH(x,D\upsilon_t,\mu_t)\cdot D\nu
+(\mu_t-1)\,\nu\bigr)\,dx
\\
&\quad
+\varepsilon\int_{\Tt^d}
\bigl(|D\upsilon_t|^{\bar\gamma-2}D\upsilon_t\cdot D\nu
+|\upsilon_t|^{\bar\gamma-2}\upsilon_t\,\nu\bigr)\,dx.
\end{aligned}
\]
By Remarks~\ref{rmk:cxty} and~\ref{rmk:implied-a}, the conventions in \eqref{a:Hat0}--\eqref{a:mDpHat0}, and the growth bounds of Assumption~\ref{onH.powergrowth} together with Lemma~\ref{lem:ADFGU}, each integrand above is (pointwise) continuous in $t$ and dominated by a fixed integrable function.  Hence, Lebesgue's dominated convergence theorem yields that
\[
t\;\longmapsto\;
\bigl\langle A_\varepsilon[\mu_t,\upsilon_t],(\eta,\nu)\bigr\rangle
\]
is continuous on $[0,1]$, as required.
\end{proof}

\begin{proposition}[Existence of a Regularized Solution]\label{thm:strongAepsi}
Under the assumptions of Theorem~\ref{thm:main-simple.power.growth}, for each $\varepsilon>0$ there exists
\[
(m_\varepsilon,u_\varepsilon)\in\mathcal K
\]
such that
\begin{equation}\label{eq:Aepsistrongineq}
\bigl\langle A_\varepsilon[m_\varepsilon,u_\varepsilon],
(\mu,\upsilon)-(m_\varepsilon,u_\varepsilon)\bigr\rangle
\;\geq\;0,
\quad\forall\,(\mu,\upsilon)\in\mathcal K.
\end{equation}
\end{proposition}
\begin{proof}
Fix $\varepsilon>0$.  
By Assumptions~\ref{onH.monotone} and \ref{onH.powergrowth}, and recalling Remarks~\ref{rmk:cxty} and~\ref{rmk:implied-a}, the operator
$A_0$ (i.e.\ \eqref{defAepsi} with $\varepsilon=0$)
is monotone on 
$
\mathcal K$.
Since $u\mapsto\|u\|_{W^{1,\bar\gamma}}^{\bar\gamma}$ is convex, its G\^ateaux derivative (the $\varepsilon$-regularization in \eqref{defAepsi}) is also monotone.  Hence, $A_\varepsilon$ is monotone and, by Proposition~\ref{prop:coercAepsi}, it is coercive.  

Since $\mathcal{K}$ is convex and closed, by Theorem~\ref{thm:monotone.abstract-weak}, there exists 
\[
(m_\varepsilon,u_\varepsilon)\in\mathcal K
\]
satisfying the weak variational inequality
\begin{equation}\label{eq:Aepsi.weak.ineq}
    \bigl\langle A_\varepsilon[ \mu,\upsilon],
(\mu,\upsilon)-(m_\varepsilon,u_\varepsilon)\bigr\rangle
\;\geq\;0,
\quad\forall\,(\mu,\upsilon)\in\mathcal K.
\end{equation}

To upgrade to the strong form \eqref{eq:Aepsistrongineq}, we invoke hemicontinuity.  By Lemma~\ref{lem:hemicont}, $A_\varepsilon$ is hemicontinuous on $\mathcal K$, therefore Remark~\ref{rmk:hemicontinuous-upgrade} applies. Hence, 
$
(m_\varepsilon,u_\varepsilon)
$
also satisfies the strong variational inequality \eqref{eq:Aepsistrongineq}. 
\end{proof}

We now show that any pair \((m,u)\) in the appropriate function spaces that satisfies the variational inequality  for \(A_\epsi\) as in \eqref{eq:Aepsistrongineq}, 
yields a strong solution of the corresponding regularized MFG system; namely, it
satisfies the Hamilton--Jacobi inequality (with equality on $\{m>0\}$) and the weak transport equation.

\begin{lemma}[Strong solution property] \label{lem:strong.is.strong}
    Under the assumptions of Theorem~\ref{thm:main-simple.power.growth}, let $\epsi\geq 0$ and  $(m, u)\in \mathcal K$ be such that
\begin{equation}\label{eq:Aepsi.strong-ineq}
    \langle A_\epsi[m,u], (\mu-m, \upsilon-u) \rangle \geq 0
\end{equation}
    for all $(\mu, \upsilon)\in L^{\bar\beta}(\Tt^d;\RR^+_0) \times W^{1,\bar\gamma}(\Tt^d)$.
    Then, we have
    \begin{align}
            & u+ H(x,Du,m)- V(x) \leq 0 \qquad \text{a.e.~in~} \Tt^d,\label{eq:HJleq0}\\
            & u+ H(x,Du,m)- V(x) = 0 \qquad \text{a.e.~in~} \{x\in\Tt^d\colon m>0\}. \label{eq:HJ0}
    \end{align}
    Moreover, the regularized transport equation holds in the weak sense; that is, 
    \begin{equation}\label{eq:strong-transport}
        \int_{\Tt^d} \big[ (m-1)\varphi + mD_pH(x,Du,m)\cdot D\varphi + \epsi|Du|^{\bar\gamma-2}Du\cdot D\varphi + \epsi |u|^{\bar\gamma-2}u\varphi\big]\, dx = 0
    \end{equation}
    for all $\varphi\in W^{1,\bar\gamma}(\Tt^d)$.
\end{lemma}

\begin{proof} 
First, we take \((\mu, \upsilon) =(m+\varphi, u)\) in \eqref{eq:Aepsi.strong-ineq} for an arbitrary $\varphi\in C^\infty(\Tt^d;\RR^+_0)$ to obtain
        \begin{equation*}
                %\label{eq:appositivedois}
                \begin{aligned}
                        \int_{\Tt^d} \big( -u - H(x, Du,m) + V(x)  \big)\varphi \, dx \geq 0.
                \end{aligned}
        \end{equation*}
From this, we conclude \eqref{eq:HJleq0}. Next, we take \((\mu, \upsilon)=\Big(m\Big(1\pm \frac{\varphi}{\Vert \varphi\Vert_\infty}\Big),u\Big)\) 
for an arbitrary non-zero $\varphi\in C^\infty(\Tt^d)$. The
previous choice ensures that $\mu\geq 0$ and, hence, it lies in the domain. 
Accordingly, using it in  \eqref{eq:Aepsi.strong-ineq} yields
\begin{equation*}
        %\label{eq:appositivetres}
        \begin{aligned}
                \pm\int_{\Tt^d} \big( -u - H(x, Du,m) + V(x)  \big)m\varphi \, dx \geq 0.
        \end{aligned}
\end{equation*}
This condition implies \eqref{eq:HJ0}.
Finally, we take \((\mu, \upsilon) =(m, u\pm\varphi)\) in \eqref{eq:Aepsi.strong-ineq} for an arbitrary $\varphi\in W^{1,\bar\gamma}(\Tt^d)$ to conclude~\eqref{eq:strong-transport}.
\end{proof}

Next, we establish uniform estimates for the solution given by Proposition~\ref{thm:strongAepsi} that will enable us to prove Theorem~\ref{thm:main-simple.power.growth}.

\begin{theorem}[Uniform estimates for \((m_\epsi,u_\epsi)\)]\label{thm:apriori-simple.power.growth}
        Under the assumptions of Theorem~\ref{thm:main-simple.power.growth}, let \(( m_\epsi,  u_\epsi) \in L^{\bar\beta}(\Tt^d;\RR^+_0) \times W^{1,\bar\gamma}(\Tt^d)\) be given by Proposition~\ref{thm:strongAepsi}. Then,  
        \begin{equation}
                \label{eq:apriori-simple.power.growth}
                \begin{aligned}
                        & \Vert m_\epsi\Vert_{L^{\bar\beta}(\Tt^d)}^{\bar\beta} + \Vert u_\epsi \Vert_{W^{1,\bar\gamma}(\Tt^d)}^{\bar\gamma}  \leq C
                \end{aligned}
        \end{equation}
for a constant $C>0$ as in Remark~\ref{rmk:onC}; in particular, it is independent of $\epsi$.
\end{theorem}

\begin{proof} 
% Throughout this proof, $C$ denotes a positive constant as in Remark~\ref{rmk:onC}, which value may change from one expression to another.
Using \eqref{eq:strong-transport} with $\varphi = u_\epsi$ and subtracting from this the equality
\[\int_{\Tt^d} m_\epsi(u_\epsi+H(x, Du_\epsi,m_\epsi)-V(x))\,
dx = 0,\]
obtained from \eqref{eq:HJ0}, we find that
\begin{equation*}
                %\label{eq:appositiveum}
        \begin{aligned}
                        0 = & \int_{\Tt^d} \Big( -u_\epsi + m_\epsi(-H(x, Du_\epsi,m_\epsi)+D_p H(x, Du_\epsi,m_\epsi)\cdot Du_\epsi ) + V(x) m_\epsi  \\
                        &\qquad  + \epsi |Du_\epsi|^{\bar\gamma} + \epsi |u_\epsi|^{\bar\gamma}\Big)\,
                        dx.
        \end{aligned}
\end{equation*}
In particular, using \eqref{eq:DpHdotpminusH-estimate1} of Lemma~\ref{lem:A1+anyother} and Young's inequality as in the proof of Proposition~\ref{prop:coercAepsi}, we get that 
 \begin{equation}\label{eq:uepsibelow}
        \begin{aligned}
                \int_{\Tt^d} u_\epsi\, dx \geq \frac{1}{C} \Vert m_\epsi\Vert_{L^{\bar\beta}(\Tt^d)}^{\bar\beta} 
        %+ {\color{ForestGreen} \epsi \Vert u_\epsi\Vert_{W^{1,\bar\gamma}(\Tt^d)}^{\bar\gamma}}  +  {\color{red} \frac{1}{c} \int_{\Tt^d}  m_\epsi |Du_\epsi|^\alpha \, dx}
        - C.  
        \end{aligned}
\end{equation}
On the other hand, integrating \eqref{eq:HJleq0} and using $-H(x,p,m)\leq C(m^{\bar\beta-1}+1)$
by \eqref{eq:assH.lower.simple}, we conclude that
        \begin{equation}
                \label{eq:uepsiabove}
                \begin{aligned}
                        & \int_{\Tt^d} u_\epsi\, dx \leq \int_{\Tt^d} \left( |V| + C\big( m_\epsi^{\bar \beta -1}+1\big)\right) \, dx.
                \end{aligned}
        \end{equation}
Then, combining \eqref{eq:uepsibelow} and \eqref{eq:uepsiabove} and using Young's inequality, we obtain
\begin{equation}\label{eq:mepsi.bound}
        \begin{aligned}
                 \Vert m_\epsi\Vert_{L^{\bar\beta}(\Tt^d)}^{\bar\beta} 
         %+ {\color{ForestGreen} \epsi \Vert u_\epsi\Vert_{W^{1,\bar\gamma}%(\Tt^d)}^{\bar\gamma}}  +   {\color{red} \int_{\Tt^d}  m_\epsi |Du_\epsi|^\alpha \, dx}
         \leq C. 
        \end{aligned}
\end{equation}
In turn, using \eqref{eq:mepsi.bound} in \eqref{eq:uepsibelow} and \eqref{eq:uepsiabove} gives that 
\[\left|\int_{\Tt^d} u_\epsi \, dx\right| \leq C.\]
This allows us to use Poincar\'e--Wirtinger's inequality to say that
\begin{equation}\label{eq:poincare-torus}
        \Vert u_\epsi \Vert_{L^{\bar\gamma}(\Tt^d)} \leq C\left(\Vert Du_\epsi\Vert_{L^{\bar\gamma}(\Tt^d)} +1\right).
\end{equation}

Next, we combine \eqref{eq:HJleq0}  with \eqref{eq:assH.lower.simple}  to obtain the pointwise inequality 
\[|Du_\epsi|^\alpha \leq C(m_\epsi^{\bar\beta-1} + |u_\epsi| + |V(x)| + 1),\]
which, raised to the power of $\bar\beta'=\frac{\bar\gamma}{\alpha}$,  yields
\begin{equation}\label{eq:Duepsiabove}
|Du_\epsi|^{\bar\gamma} \leq C(m_\epsi^{\bar\beta} + |u_\epsi|^{\bar\beta'} + |V(x)|^{\bar\beta'}+1).
\end{equation}
Integrating \eqref{eq:Duepsiabove} over $\Tt^d$ and using \eqref{eq:mepsi.bound} and the continuous embedding of $L^{\bar\gamma}$ into $L^{\bar\beta'}$, we find that
\begin{equation*}
\Vert Du_\epsi\Vert_{L^{\bar\gamma}(\Tt^d;\RR^d)}^{\bar\gamma} \leq C\left(\Vert u_\epsi\Vert_{L^{\bar\gamma}(\Tt^d)}^{\bar\beta'} + 1\right).
\end{equation*}
Consequently, recalling \eqref{eq:poincare-torus}, we have that
\begin{equation*}
\Vert Du_\epsi\Vert_{L^{\bar\gamma}(\Tt^d;\RR^d)}^{\bar\gamma} \leq C\left(\Vert Du_\epsi\Vert_{L^{\bar\gamma}(\Tt^d)}^{\bar\beta'} + 1\right),
\end{equation*}
from which we conclude that
\begin{equation}\label{eq:Duepsi.bound}
\Vert Du_\epsi\Vert_{L^{\bar\gamma}(\Tt^d;\RR^d)} \leq C
\end{equation}
because ${\bar\gamma}>\bar\beta'$.
Finally, we obtain \eqref{eq:apriori-simple.power.growth} from \eqref{eq:Duepsi.bound}, \eqref{eq:poincare-torus}, and \eqref{eq:mepsi.bound}.
\end{proof}

We now prove Theorem~\ref{thm:main-simple.power.growth} by considering the limit 
$\epsi\to 0$.

\begin{proof}[Proof of Theorem~\ref{thm:main-simple.power.growth}]
By \eqref{eq:Aepsi.weak.ineq} in the proof of Proposition~\ref{thm:strongAepsi}, we can  find a sequence  \((m_{\epsi_k}, u_{\epsi_k})_k \subset L^{\bar\beta}(\Tt^d;\RR^+_0) \times W^{1,\bar\gamma}(\Tt^d)\), with \(\epsi_k\to0^+\),
satisfying 
\begin{equation}
\label{eq:existweaksol}
\begin{aligned}
0 &\leq \langle A_{\epsi_k}  [\mu,
\upsilon] , (\mu, \upsilon)-(m_{\epsi_k},
u_{\epsi_k})\rangle \\
&= \int_{\Tt^d} \Big(  -\upsilon- H(x, D\upsilon,\mu) + V(x)  \Big)(\mu \, -m_{\epsi_k})\, dx
\\ &\quad+\int_{\Tt^d} \Big( \mu D_p H(x, D\upsilon,\mu) \cdot(D\upsilon - Du_{\epsi_k})  + (\mu-1)( \upsilon-u_{\epsi_k}) \Big)
\, dx
\\&\quad+\epsi_k\int_{\Tt^d}
 \Big( |D\upsilon|^{\bar\gamma-2} D\upsilon \cdot D( \upsilon-u_{\epsi_k}) 
+|\upsilon|^{\bar\gamma-2}\upsilon( \upsilon-u_{\epsi_k})\Big)\,dx
\end{aligned}
\end{equation}
for every \((\mu, \upsilon)\in L^{\bar\beta}(\Tt^d;\RR^+_0) \times W^{1,\bar\gamma}(\Tt^d) \) and \(k\in\Nn\). Moreover, by Theorem~\ref{thm:apriori-simple.power.growth}, we can find \((m,u)\in L^{\bar\beta}(\Tt^d;\RR^+_0)\times W^{1,\bar\gamma}(\Tt^d)\) such that, up to a subsequence that we do not relabel,
\begin{equation}
\label{eq:compactness}
\begin{aligned}
& (m_{\epsi_k}, u_{\epsi_k}) \rightharpoonup (m,u)\quad \text{weakly in } L^{\bar\beta}(\Tt^d;\RR^+_0)\times W^{1,\bar\gamma}(\Tt^d).
\end{aligned}
\end{equation}
Letting \(k\to\infty\) in \eqref{eq:existweaksol}, we conclude from \eqref{eq:compactness} combined with H\"older's inequality that
\begin{equation*}%\label{eq:weakpositive}
\begin{aligned}
0 \leq& \int_{\Tt^d}
\Big(  -\upsilon- H(x, D\upsilon,\mu) + V(x)  \Big)(\mu
\, -m)\,dx\\
&\quad+\int_{\Tt^d} \Big( \mu D_p H(x, D\upsilon,\mu) \cdot(D\upsilon - Du)  + (\mu-1)( \upsilon-u) \Big)
\,dx\\
=&\,\langle A_0[\mu,
\upsilon], (\mu, \upsilon)-(m,u) \rangle  
\end{aligned}
\end{equation*}
for every \((\mu, \upsilon)\in L^{\bar\beta}(\Tt^d;\RR^+_0) \times W^{1,\bar\gamma}(\Tt^d) \). Hence, \((m,u)\)  in \eqref{eq:compactness} is at least a weak solution to \eqref{smfg} in the sense of Definition~\ref{def:weak}. Furthermore, because $A_0$ is hemicontinuous, we get
\[\langle A_0[m,
u], (\mu, \upsilon)-(m,u) \rangle \geq 0\]
for every \((\mu, \upsilon)\in L^{\bar\beta}(\Tt^d;\RR^+_0) \times W^{1,\bar\gamma}(\Tt^d) \) by applying the idea in Remark~\ref{rmk:hemicontinuous-upgrade}. We then conclude the proof by Lemma~\ref{lem:strong.is.strong} with $\epsi=0$.
\end{proof}

\section{Proof of Theorem~\ref{thm:main-congestion.growth}} 
\label{sec:proof-congestion}

This section establishes Theorem~\ref{thm:main-congestion.growth}, adapting the proof from the power-growth case to accommodate Hamiltonians with congestion terms. The analysis requires two main modifications. First, the singular growth conditions in Assumption~\ref{onH.congestiongrowth} present a significant challenge, as they prevent the operator $A_\epsi$ 
defined in \eqref{defAepsi} 
from being well-defined on the entire space $L^{\bar\beta}(\Tt^d; \RR^+_0)\times W^{1,\bar\gamma}(\Tt^d)$ due to potential integrability issues. To overcome this, we must restrict the operator to a carefully constructed convex subset, $\mathcal K$. Second, this restriction introduces a further subtlety: since the solution is not guaranteed to lie within $\mathcal K$, the standard hemicontinuity approach (Remark~\ref{rmk:hemicontinuous-upgrade}) is no longer applicable. The core of the proof is therefore a dedicated lemma (Lemma~\ref{lem:weak.is.strong}) that directly establishes the strong solution property.

We begin by defining the domain:
\begin{equation}\label{def-domainK}
    \mathcal{K} := \left\{(\mu,\upsilon)\in L^{\bar\beta}(\Tt^d)\times W^{1,\bar\gamma}(\Tt^d)\colon \essinf \mu > 0 \quad\text{and}\quad \frac{|D\upsilon|^{\alpha(1+\frac{\tau}{\beta})}}{\mu^\tau+1} \in L^{\bar\beta'}(\Tt^d)\right\}.
\end{equation}
We note that while
\[
\bar{\mathcal K}=L^{\bar\beta}(\Tt^d; \RR^+_0) \times W^{1,\bar\gamma}(\Tt^d), \]
working in $\mathcal K$ avoids singularities when $\mu=0$
since $\essinf \mu>0$.

Before proceeding, we must verify two fundamental properties of this set: that it is convex, making it suitable for our variational framework, and that the operator $A_\epsi$ is in fact well-defined on it. We establish these in the following two lemmas.

\begin{lemma}
The set $\mathcal{K}$ in \eqref{def-domainK} is convex.
\end{lemma}
\begin{proof}
We start by noting  that if $(\mu, \upsilon)\in \mathcal{K}$ and $\essinf \tfrac{\mu'}{\mu} > 0$ for some $\mu' \in L^{\bar\beta}(\Tt^d)$, then we have $(\mu', \upsilon)\in \mathcal{K}$ as well. In particular, given $(\mu_1, \upsilon_1)\in \mathcal{K}$ and $(\mu_2, \upsilon_2)\in \mathcal{K}$, and some $0<t<1$, we get that
\[(t\mu_1+(1-t)\mu_2,\ \upsilon_1)\in\mathcal{K}, \qquad (t\mu_1+(1-t)\mu_2,\ \upsilon_2)\in\mathcal{K}.\]
Consequently, also using  the convexity of the function $|\cdot|^{\alpha(1+\frac{\tau}{\beta})}$ and of the space $L^{\bar\beta'}(\Tt^d)$, one finds that
\[(t\mu_1+(1-t)\mu_2,\ t\upsilon_1+(1-t)\upsilon_2)\in\mathcal{K}.\]
Because  $(\mu_1, \upsilon_1)$, $(\mu_2, \upsilon_2) \in \mathcal{K}$ and $t\in (0,1)$ were taken arbitrarily, we conclude that $\mathcal{K}$ is convex.
\end{proof}

\begin{lemma}\label{lem:Aepsi.congestion.well-defined}
Let $A_\epsi$ be given by \eqref{defAepsi}. Then, $A_\epsi$ is a well-defined map from $\mathcal K$ to $L^{\bar\beta'}(\Tt^d) \times W^{-1,\bar\gamma'}(\Tt^d)$.
\end{lemma}
\begin{proof}
We observe that $(\mu,\upsilon)\in \mathcal{K}$ implies
\[\left(\frac{1}{\mu} + \frac{1}{\mu^\tau}\right)(|D\upsilon|^{\alpha(1+\frac{\tau}{\beta})} + 1) \in L^{\bar\beta'}(\Tt^d)\]
because $(\tfrac{1}{\mu} + \tfrac{1}{\mu^\tau})(\mu^\tau +1) \leq 3(\essinf \mu)^{-1} + 3$ a.e.~on $\Tt^d$. Hence, $H(x,D\upsilon,\mu)\in L^{\bar\beta'}(\Tt^d)$ for all $(\mu,\upsilon)\in \mathcal{K}$ by \eqref{eq:assH.lower.congestion} and \eqref{eq:boundsH+}.
These integrability properties, together with \eqref{eq:boundsmDpH+} and the fact that $V\in L^{\bar\beta'}(\Tt^d)$, confirm that each component of the operator is well-defined. Therefore, $A_\epsi$ maps $\mathcal K$ to the target space $L^{\bar\beta'}(\Tt^d) \times W^{-1,\bar\gamma'}(\Tt^d)$.
\end{proof}

Having established that $A_\epsi$ is a well-defined operator on the convex set $\mathcal K$, we proceed to verify the conditions of Theorem~\ref{thm:monotone.abstract-weak}. The monotonicity of $A_\epsi$ is an immediate consequence of Assumption~\ref{onH.monotone}, as in the previous section. Coercivity also holds, following the same line of argument as in Proposition~\ref{prop:coercAepsi}. With these conditions satisfied, the abstract theorem guarantees the existence of a solution to the regularized problem, which we state in the following proposition.
\begin{proposition}[Existence of a Regularized Solution]\label{thm:weakAepsi.congestion}
Under the assumptions of Theorem~\ref{thm:main-congestion.growth}, let $A_\epsi$ be defined as in Lemma \ref{lem:Aepsi.congestion.well-defined}. Then, for every $\epsi>0$, there exists $(m_\epsi, u_\epsi)\in L^{\bar\beta}(\Tt^d;\RR^+_0) \times W^{1,\bar\gamma}(\Tt^d)$  satisfying
    \[\langle A_\epsi[\mu,\upsilon], (\mu-m_\epsi, \upsilon-u_\epsi) \rangle \geq 0\]
    for all $(\mu, \upsilon)\in \mathcal{K}$.
\end{proposition}
\begin{proof}
    As mentioned above, the conclusion follows from Theorem~\ref{thm:monotone.abstract-weak}, once we prove that $A_\epsi$ satisfies the conditions therein. First, as in Proposition~\ref{thm:strongAepsi}, Assumption~\ref{onH.monotone} ensures that $A_\epsi$ is monotone. Thus, it remains to establish the coercivity of $A_\epsi$. For this, we note that the estimates in \eqref{eq:coercive.Aepsi} are also satisfied in the present
setting because  \eqref{eq:DpHdotpminusH-estimate1}
holds by Assumption~\ref{onH.monotone}
  and  \eqref{eq:assH.upper.congestion},
 $|H(x,0,1)| \leq C$ by \eqref{eq:assH.upper.congestion} and \eqref{eq:assH.lower.congestion},  $|D_pH(x,0,1)|
\leq C$  by \eqref{eq:assH.DpH.upper.congestion},
and \eqref{eq:assH.lower.simple} {holds}
by Remark~\ref{rem:hierarchy} (noting that \eqref{eq:assH.lower.weak} and \eqref{eq:assH.lower.simple}
are equivalent). Hence, $A_\epsi$ is also coercive, which concludes
the proof. 
\end{proof}

As anticipated, proving the strong solution property for the congestion case requires a dedicated argument, since the solution $$(m_\epsi, u_\epsi)$$ is not guaranteed to lie in the operator's domain $\mathcal K$. The following key lemma provides this argument, showing that any solution to the weak variational inequality is a strong solution.

\begin{lemma}[Strong solution property]\label{lem:weak.is.strong} Under the assumptions of Theorem~\ref{thm:main-congestion.growth}, let $\epsi\geq 0$ and  $(m, u)\in L^{\bar\beta}(\Tt^d;\RR^+_0) \times W^{1,\bar\gamma}(\Tt^d)$ be such that
\begin{equation}
\label{eq:weak.epsi.cong}
\begin{aligned}
 \langle A_\epsi[\mu,\upsilon], (\mu-m, \upsilon-u) \rangle \geq 0
\end{aligned}
\end{equation}
    for all $(\mu, \upsilon)\in \mathcal{K}$.
    Then, we have
    \begin{align}
            & u+ H(x,Du,m)- V(x) \leq 0 \qquad \text{a.e.~in~} \Tt^d, \label{eq:congest.HJleq0}\\
            & u+ H(x,Du,m)- V(x) = 0 \qquad \text{a.e.~in~} \{x\in\Tt^d\colon m>0\}, \label{eq:congest.HJ0}
    \end{align}
and
    \begin{equation}\label{eq:congest.transport}
        \int_{\Tt^d} \Big[ (m-1)\varphi + mD_pH(x,Du,m)\cdot D\varphi + \epsi|Du|^{\bar\gamma-2}Du\cdot D\varphi + \epsi |u|^{\bar\gamma-2}u\varphi \Big]\, dx = 0
    \end{equation}
    for all $\varphi\in W^{1,\bar\gamma}(\Tt^d)$.
Furthermore, 
\begin{equation}
\label{eq:reg.pair.sol}
\begin{aligned}
H(x,Du, m)\in L^{\bar \beta'}(\Tt^d)\quad \text{and}\quad\frac{|Du|^{\alpha(1+\frac{\tau}{\beta})}}{m^\tau+1}\in
L^{\bar\beta'}(\Tt^d).
\end{aligned}
\end{equation}
    
\end{lemma}

\begin{proof} 
The proof proceeds in three main steps. For the properties of the Hamilton--Jacobi equation,  \eqref{eq:congest.HJleq0} and \eqref{eq:congest.HJ0}, we construct specific test functions $\mu$ by defining them pointwise as minimal ``corrections'' to $m$. Since these test functions may not be in $\mathcal K$, we use an approximation argument ($\mu_\delta = \max\{ \mu, \delta\}$) and pass to the limit.  To establish the transport equation \eqref{eq:congest.transport}, we use a different, two-stage approach. First, we test the variational inequality with an approximating pair $(m_\delta, u+\varphi)$, where $m_\delta=\max\{ m,\delta\}$ and $\varphi$ is a smooth function, 
to derive an inequality for the limit objects. Second, 
we upgrade this inequality to an equality using a first-order variation argument (testing with $t\varphi$).

\medskip
\noindent{\bf Step 1:} Proof of the Hamilton--Jacobi inequality \eqref{eq:congest.HJleq0}.
\medskip

Our strategy is to define a test function $\mu$ as the minimal ``correction'' to $m$ required to make the Hamilton--Jacobi expression less than an arbitrary positive tolerance $k$. The construction is designed such that if $m$ itself fails to satisfy this condition, then $\mu$ must be strictly greater than $m$. The core of the argument is to then use the variational inequality \eqref{eq:weak.epsi.cong} to show that this correction is impossible (i.e., that $\mu=m$ a.e.), which forces the conclusion that $m$ must have satisfied the desired inequality in the first place.

To this end, we fix some $k>0$ and define, for a.e.~$x\in\Tt^d$,
    \begin{equation}\label{eq:mu-definition}
        \mu(x):= \inf\{\mu_0\in\RR^+\colon \mu_0\geq m(x),\, u(x)+H(x,Du(x),\mu_0) - V(x)\leq k\}.
    \end{equation}
    First, we note that the set on the right-hand side of \eqref{eq:mu-definition} is non-empty by \eqref{eq:boundsH+}, and  the function $\mu\colon\Tt^d\to \RR^+_0$ is measurable in view of  the identity  $\{x\colon \mu\leq a\} = \{x\colon m\leq a\} \cap \{x\colon u+H(x,Du,a)-V(x)\leq k\}$ for every $a\in\RR^+$. Secondly, when $H(x,Du,\mu)$ is evaluated in the set $\{x\colon \mu=0\}$ in the sense of (\ref{a:Hat0}), it takes a finite value because of the upper bound in~\eqref{eq:mu-definition}. Moreover, we have the following identity
    \begin{equation}\label{eq:mu-minus-m-formula}
        (\mu-m)\big(-k+u+H(x,Du,\mu)-V(x)\big) = 0.
    \end{equation}
To see why this holds, consider the two cases for the first factor.   If $\mu(x) = m(x)$, the identity holds trivially. If $\mu(x) > m(x)$, the equality $u+H(x,Du,\mu)-V(x) = k$ must hold. Indeed, we know from the definition that $u+H(x,Du,\mu)-V(x) \leq k$. If this inequality were strict, the continuity of $H$ would allow us to find a smaller value $\mu(x)-\varepsilon > m(x)$ that still satisfies the condition, contradicting the definition of $\mu(x)$ as the infimum. Thus, the identity \eqref{eq:mu-minus-m-formula} holds.

Next, we establish the necessary integrability properties of $\mu$. In the set where the second factor in \eqref{eq:mu-minus-m-formula} vanishes, we have $u + H(x, Du, \mu) - V(x) = k$. The estimate \eqref{eq:boundsH+} then gives
    \[0 \leq k = u + H(x, Du, \mu) - V(x) \leq C\left(\frac{1}{\mu} + \frac{1}{\mu^\tau}\right)(|Du|^{\alpha(1+\frac{\tau}{\beta})} +1) - \frac{1}{C}\mu^{\beta} +|u|+|V(x)|.\]
    Here, we can rearrange the inequality between the extreme terms by first moving the negative terms across, then multiplying by $(1/\mu + 1/\mu^\tau)^{-1}$ and combining with the estimates
     \[\frac{\mu^{\beta+\tau}-1}{2} \leq \mu^\beta\left(\frac{1}{\mu} + \frac{1}{\mu^\tau}\right)^{-1} \qquad \text{and}\qquad \left(\frac{1}{\mu} + \frac{1}{\mu^\tau}\right)^{-1} \leq \mu^\tau,\]
    and, finally, redefining $C$, to obtain
    \[\frac{\mu^{\beta+\tau}-1}{2} \leq C\big(|Du|^{\alpha(1+\frac{\tau}{\beta})}+1 + \mu^\tau(|u|+|V(x)|)\big).\]
    Then, we can use Young's inequality and other standard rearrangements to conclude
    \[\mu\leq C|Du|^{\alpha/\beta}+C(|u|+|V(x)|)^{1/\beta} +C
    \in L^{\bar\beta}(\Tt^d).\]
    On the other hand, in the set where the first term of the product in \eqref{eq:mu-minus-m-formula} vanishes, we have $\mu = m \in L^{\bar\beta}(\Tt^d)$. Consequently, $\mu\in L^{\bar\beta}(\Tt^d;\RR^+_0)$. Furthermore, \eqref{eq:assH.lower.congestion} and \eqref{eq:mu-definition} yield the estimate
    \begin{equation}\label{eq:H-at-Du-mu.bounds}
      \frac{1}{C}\left(\frac{|Du|^{\alpha(1+\frac{\tau}{\beta})}}{\mu^\tau+1}\right) - C(\mu^\beta+1) \leq H(x,Du,\mu) \leq k + |u| + |V(x)|,  
    \end{equation}
    which shows at once that $H(x,Du,\mu)\in L^{\bar\beta'}(\Tt^d)$ and $\frac{|Du|^{\alpha(1+\frac{\tau}{\beta})}}{\mu^\tau+1} \in L^{\bar\beta'}(\Tt^d)$.

Now we proceed with the approximation argument. While the identity \eqref{eq:mu-minus-m-formula} is central, we cannot use $(\mu, u)$ as a test function directly, as it is not guaranteed to be in $\mathcal K$ (we may have $\essinf \mu = 0$). To circumvent this, we define an admissible approximation $\mu_\delta(x) := \max\{\mu(x),\delta\}$ for $\delta > 0$. 
Recalling \eqref{def-domainK}, observe that $(\mu_\delta,u)\in\mathcal{K}$ because $\essinf\mu_\delta \geq \delta > 0$ and the bounds above yield $\mu_\delta\in L^{\bar\beta}(\Tt^d)$ and $\frac{|Du|^{\alpha(1+\frac{\tau}{\beta})}}{\mu_\delta^\tau+1} \in L^{\bar\beta'}(\Tt^d)$. Therefore, we can use \eqref{eq:weak.epsi.cong}  to
conclude that
    \begin{equation}\label{eq:hjb-test-mu.delta}
        \int_{\Tt^d} (u+H(x,Du,\mu_\delta)-V(x))(\mu_\delta-m)\,
dx \leq 0.
    \end{equation}
    Moreover, the integrability bounds above on $\mu_\delta$ and $H(x,Du,\mu_\delta)$, hence on the two multiplied terms of the integrand in the last estimate, are uniform for $\delta\leq 1$. This allows us to use the dominated convergence theorem to conclude that
    \[k\int_{\Tt^d}(\mu-m) \, dx= \int_{\Tt^d} (u+H(x,Du,\mu)-V(x))(\mu-m)\,
dx \leq 0,\]
    where we also used \eqref{eq:mu-minus-m-formula}. 
  Since $k>0$ and $\mu \geq m$ by definition, we must have $\mu=m$ a.e.. From the definition of $\mu$, $u+H(x,Du,m)-V(x)\leq k$ a.e.~by~\eqref{eq:mu-definition}. Because $k>0$ is arbitrary,  we achieve \eqref{eq:congest.HJleq0}.

Finally, we note that \eqref{eq:H-at-Du-mu.bounds}
combined with $\mu=m$ yields \eqref{eq:reg.pair.sol}.

\medskip
\noindent{\bf Step 2:} Proof of the equality \eqref{eq:congest.HJ0} on $\{m>0\}$.    
\medskip 

 The approach for proving this equality mirrors that of Step 1, but requires a test function designed to prove the inequality in the opposite direction. To do this, we will construct a competitor density $\mu$ that is smaller than $m$ wherever $m$ fails to satisfy the condition { $u+H(x,Du,m)-V(x) \geq 0$}. This is achieved by defining $\mu$ as the minimum of $m$ and a new infimum targeting the tolerance $-k$.

    To this end, we fix $k>0$ and  set
    \begin{equation}\label{eq:mu-re.definition}
        \mu(x):=
        \min\big\{m(x), \inf\{\mu_0\in\RR^+\colon u(x)+H(x,Du(x),\mu_0)-V(x) \leq -k\}\big\}.
    \end{equation}
    First, we observe the crucial bound
    \begin{equation}\label{eq:mu.redef.up}
        u+H(x,Du,\mu)-V(x)\leq 0,
    \end{equation}
    which holds in the set $\{x\colon \mu < m\}$ by \eqref{eq:mu-re.definition} and in the set $\{x\colon \mu = m\}$ by \eqref{eq:congest.HJleq0}. Second,  the inequality
    \begin{equation}\label{eq:mu.redef-minus-m.bound}
        (\mu-m)\big(k+u+H(x,Du,\mu)-V(x)\big) \geq 0
    \end{equation}
    holds because the first term in the product is non-positive, while the second term is non-positive in the set where the first term does not vanish by \eqref{eq:mu-re.definition}. Note that \eqref{eq:mu.redef-minus-m.bound} is analogous to \eqref{eq:mu-minus-m-formula}, but unlike the latter, it may have strict inequality in the set $\{x\colon m>\mu=0\}$. Thirdly, $\mu\in L^{\bar\beta}(\Tt^d)$  because $\mu\leq m$. Finally, 
    the bounds in \eqref{eq:H-at-Du-mu.bounds}  hold in the present
setting  by \eqref{eq:mu.redef.up} and \eqref{eq:assH.lower.congestion}, proving that $H(x,Du,\mu)\in L^{\bar\beta'}(\Tt^d)$ and $\frac{|Du|^{\alpha(1+\frac{\tau}{\beta})}}{\mu^\tau+1} \in L^{\bar\beta'}(\Tt^d)$ as before.

With these properties established, the remainder of the proof proceeds in a similar way to Step 1. We first construct an approximation
$(\mu_\delta,u)\in \mathcal{K}$ with $\mu_\delta(x) := \max\{\mu(x),\delta\}$.
We use $\mu_\delta$ as an admissible approximation of our ideal test function $\mu$, apply the variational inequality obtaining 
\eqref{eq:hjb-test-mu.delta}.
 Moreover, the integrability bounds are uniform for $\delta\leq 1$, allowing us to use the dominated convergence theorem to
 take the limit $\delta\to 0$ and 
 get
    \[k\int_{\Tt^d}(m-\mu)\, dx \leq \int_{\Tt^d} (u+H(x,Du,\mu)-V(x))(\mu-m)\,
dx \leq 0,\]
    where we also used \eqref{eq:mu.redef-minus-m.bound}.
    Since $k>0$, the last estimate ensures that $\mu=m$ a.e..
Hence, we must have $u+H(x,Du,m)-V(x) \geq -k$ in the set $\{x \colon m>0\}$. If the inequality were false, the infimum in the definition \eqref{eq:mu-re.definition} would be strictly less than $m(x)$ (by continuity and monotonicity of $H$), which would force $\mu(x) < m(x)$, a contradiction to $\mu=m$. Noting, as before, that $k>0$ is arbitrary, we conclude \eqref{eq:congest.HJ0}.

\medskip
\noindent{\bf Step 3:} Proof of the transport equation \eqref{eq:congest.transport}.
\medskip

The  strategy to prove the transport equation is to first test the variational inequality with an approximating pair $(m_\delta, u+\varphi)$, for a smooth function $\varphi$, and take the limit as $\delta \to 0$. This yields a variational inequality for the limit objects. We then upgrade this inequality to the desired equality by employing a first-order variation argument, which involves replacing $\varphi$ with $t\varphi$ and analyzing the limit as the parameter $t \to 0$.

 By density of $C^\infty(\Tt^d)$ in $W^{1,\bar\gamma}(\Tt^d)$
   and the bounds in \eqref{eq:boundsmDpH+},
    it is enough to prove \eqref{eq:congest.transport} for any $\varphi\in C^\infty(\Tt^d)$.
    We then set $m_\delta:=\max\{m, \delta\}$ for $\delta>0$ and observe that $(m_\delta, u+\varphi)\in \mathcal{K}$ for arbitrary $\varphi\in C^\infty(\Tt^d)$ by \eqref{eq:reg.pair.sol} and because
    \[\frac{|Du+D\varphi|^{\alpha(1+\frac{\tau}{\beta})}}{m_\delta^\tau+1} \leq 2^{{\alpha(1+\frac{\tau}{\beta})}-1}\left(\frac{|Du|^{\alpha(1+\frac{\tau}{\beta})}}{m^\tau+1}
+|D\varphi|^{\alpha(1+\frac{\tau}{\beta})}\right).\]
     Therefore, using \eqref{eq:weak.epsi.cong} with $(\mu,\upsilon)=(m_\delta,u+\varphi)$,   we have
    \begin{equation*}
    \begin{aligned}
        & \int_{\Tt^d} \big[(m_\delta-1)\varphi + m_{\delta}D_pH(x,Du+D\varphi,m_\delta)\cdot D\varphi\big]
\,dx \\
        &\quad + \int_{\Tt^d} \Big[\epsi|Du+D\varphi|^{\bar\gamma-2}(Du+D\varphi)\cdot D\varphi + \epsi|u+\varphi|^{\bar\gamma-2}(u+\varphi)\varphi\Big]
\,dx \\
        & \qquad\qquad
        \geq \int_{\Tt^d} (u+\varphi+H(x,Du+D\varphi,m_\delta)-V(x))(m_\delta-m)\,
dx \\
        & \qquad\qquad
        \geq \int_{\Tt^d} (u+\varphi-C(m_\delta^\beta+1)-V(x))(m_\delta-m)\,
dx,
    \end{aligned}
    \end{equation*}
where we used \eqref{eq:assH.lower.congestion} in the last step. The integrands in the top and the bottom lines satisfy integrability bounds that are uniform for $\delta\leq 1$ because of \eqref{eq:assH.DpH.upper.congestion} and $m_\delta\leq m+1$; thus, we can use dominated convergence theorem in the limit $\delta\to 0$ for the inequality between the extreme terms to get
    \begin{equation}\label{eq:transport-test-m.delta}
    \begin{aligned}
        & \int_{\Tt^d} \big[(m-1)\varphi + mD_pH(x,Du+D\varphi,m)\cdot D\varphi\big]
\,dx \\
        &\quad + \int_{\Tt^d} \Big[\epsi|Du+D\varphi|^{\bar\gamma-2}(Du+D\varphi)\cdot D\varphi + \epsi|u+\varphi|^{\bar\gamma-2}(u+\varphi)\varphi\Big]
\,dx  \geq 0.
    \end{aligned}
    \end{equation}
    Note that the expression $mD_pH(x,Du+D\varphi,m)$ here is evaluated in the set $\{x\colon m=0\}$ in the sense of (\ref{a:mDpHat0}). Also notice how we avoided taking the limit of the middle integral because $H(x,Du+D\varphi,m)$ may equal $+\infty$ at some points with $m=0$.
    %and even $H(x,Du+D\varphi,m_\delta)m_\delta$ may not converge at those points, although it remains bounded. 
 Next, we replace $\varphi$ in \eqref{eq:transport-test-m.delta} by $t\varphi$ for some $t\neq 0$ and cancel through $t$ to find that
\begin{equation*}%\label{eq:def.iota}
\begin{aligned}
       \iota(t) = & \int_{\Tt^d} \big[(m-1)\varphi + mD_pH(x,Du+tD\varphi,m)\cdot D\varphi\big]
\,dx \\
        & + \int_{\Tt^d} \Big[\epsi|Du+tD\varphi|^{\bar\gamma-2}(Du+tD\varphi)\cdot D\varphi + \epsi|u+t\varphi|^{\bar\gamma-2}(u+t\varphi)\varphi\Big]
\,dx 
\end{aligned}
\end{equation*}
is such that $\iota(t) \geq 0$ for $t>0$, and  $\iota(t) \leq 0$ for $t<0$. Moreover,  the integrand defining $\iota$ is continuous in $t$ for a.e.~$x\in\Tt^d$ by (\ref{a:mDpHat0}), and it satisfies integrability bounds that are uniform for $-1\leq t\leq 1$. Hence, we use dominated convergence theorem in the limit $t\to 0$ to  \eqref{eq:congest.transport}.
\end{proof}

\begin{theorem}[Uniform estimates for \((m_\epsi,u_\epsi)\)]\label{thm:apriori-congestion.growth}
        Under the assumptions of Theorem~\ref{thm:main-congestion.growth}, let \(( m_\epsi,  u_\epsi) \in L^{\bar\beta}(\Tt^d;\RR^+_0) \times W^{1,\bar\gamma}(\Tt^d)\) be given by Proposition~\ref{thm:weakAepsi.congestion}. Then,  
        \begin{equation*}
                %\label{eq:apriori-congestion.growth}
                \begin{aligned}
                        & \Vert m_\epsi\Vert_{L^{\bar\beta}(\Tt^d)}^{\bar\beta} + \Vert u_\epsi \Vert_{W^{1,\bar\gamma}(\Tt^d)}^{\bar\gamma}  \leq C
                \end{aligned}
        \end{equation*}
for a constant $C>0$ as in Remark~\ref{rmk:onC}; in particular, it is independent of $\epsi$.
\end{theorem}

\begin{proof}
We follow the proof of Theorem~\ref{thm:apriori-simple.power.growth} line by line, only replacing the use of Lemma~\ref{lem:strong.is.strong} by Lemma~\ref{lem:weak.is.strong} and noting that the use of \eqref{eq:assH.lower.simple} is still valid since it is implied by \eqref{eq:assH.lower.congestion} as explained in Remark~\ref{rem:hierarchy}.
\end{proof}

\begin{proof}[Proof of Theorem~\ref{thm:main-congestion.growth}]
We proceed as in the proof of Theorem~\ref{thm:main-simple.power.growth}.
By Proposition~\ref{thm:weakAepsi.congestion}, we can  find a sequence  \((m_{\epsi_k}, u_{\epsi_k})_k \subset L^{\bar\beta}(\Tt^d;\RR^+_0) \times W^{1,\bar\gamma}(\Tt^d)\), with \(\epsi_k\to0^+\),
satisfying
\begin{equation}\label{eq:existweaksol.congestion}
    0 \leq \langle A_{\epsi_k}  [\mu,
\upsilon] , (\mu, \upsilon)-(m_{\epsi_k},
u_{\epsi_k})\rangle
\end{equation}
for every \((\mu, \upsilon)\in \mathcal{K} \) and \(k\in\Nn\). Then, by Theorem~\ref{thm:apriori-congestion.growth}, we can find \((m,u)\in L^{\bar\beta}(\Tt^d;\RR^+_0)\times W^{1,\bar\gamma}(\Tt^d)\) such that, up to a subsequence that we do not relabel,
\begin{equation}
\label{eq:compactness.congestion}
\begin{aligned}
& (m_{\epsi_k}, u_{\epsi_k}) \rightharpoonup (m,u)\quad \text{weakly in } L^{\bar\beta}(\Tt^d;\RR^+_0)\times W^{1,\bar\gamma}(\Tt^d).
\end{aligned}
\end{equation}
Letting \(k\to\infty\) in \eqref{eq:existweaksol.congestion}, we deduce from \eqref{eq:compactness.congestion} combined with H\"older's inequality that
\begin{equation*}
0 \leq \langle A_0[\mu,
\upsilon], (\mu, \upsilon)-(m,u) \rangle 
\end{equation*}
for every \((\mu, \upsilon)\in \mathcal{K} \). Finally, we conclude the proof by Lemma~\ref{lem:weak.is.strong} with $\epsi = 0$.
\end{proof}

\section{Proof of Theorem~\ref{thm:main-weak.power.growth}} 
\label{sec:proof-weak-growth}

In this section, we establish the last of our results: the existence of weak solutions for mean-field games under the minimal growth conditions of Assumption~\ref{onH.weakgrowth}, as stated in Theorem~\ref{thm:main-weak.power.growth}. The core challenge in this setting is that the Hamiltonian lacks the growth properties required to ensure the MFG operator is well-defined. Our approach  regularizes the Hamiltonian itself through infimal convolution. This technique creates a new, regularized Hamiltonian, $H^\epsi$, that satisfies the necessary growth bounds while preserving the essential monotonicity structure and uniform lower bounds of the original problem. The proof proceeds in several stages: first, we define $H^\epsi$ via infimal convolution and prove its key analytical properties. 
Second, using this regularized Hamiltonian, we solve the associated variational inequality to obtain an approximate solution $(m_\epsi, u_\epsi)$. Third, we derive uniform a priori estimates for these solutions, independent of the regularization parameter $\epsi$. In the final stage, we pass to the limit as $\epsi \to 0$ to recover a weak solution to the original mean-field game.

Suppose that $H$ satisfies Assumption~\ref{onH.monotone} and Assumption~\ref{onH.weakgrowth}.  Recall that Assumption~\ref{onH.weakgrowth} provides a coercive lower bound ($H(x,p,m)\geq \frac 1 C |p|^\alpha - C(m^\beta+1)$) and a control on the negative part of the Hamiltonian ($H(x,0,m) \leq -\frac{1}{C}m^\beta+C$), but critically, it imposes no quantitative upper bound on $H$ with respect to the momentum $p$. This lack of an upper bound is the primary motivation for the infimal convolution.

The idea behind the infimal convolution is clear in the Lagrangian picture: indeed, the lower bound \eqref{eq:assH.lower.weak} translates to an upper bound on the Lagrangian $L$, while the lack of an effective upper bound on $H$ in Assumption~\ref{onH.weakgrowth} means the lack of such a lower bound on $L$. Then, infimal convolution in $H$ is equivalent to the addition to $L$ of a power-growth term of scale $\epsi$, to introduce a lower bound of the desired type.
To implement this regularization, we consider the strictly convex function $K\colon \RR^d\to\RR$ defined as
\[K(q) := |q| + |q|^\alpha.\]
For any $\epsi>0$, we define a regularized Hamiltonian as
\begin{equation}\label{eq:defHdelta}
    H^\epsi(x,p,m):= \inf_q \bigg\{H(x,p-q,m)+\frac{1}{\epsi}K(q)\bigg\}
\end{equation}
for a.e.~$x\in\Tt^d$; that is, $H^\epsi$ is the infimal convolution of $H$ with $(1/\epsi)K$ at fixed $x$ and $m$.
 We note that for each fixed $x \in \mathbb{T}^{d}$ and $m \in \mathbb{R}^{+}$, \eqref{eq:assH.lower.weak} gives for sufficiently large $q$ that
\[H(x,p-q,m) + \frac{1}{\epsi}K(q) \geq \frac{1}{C}|p-q|^{\alpha} - C(m^\beta + 1) \geq H(x,0,m) + \frac{1}{\epsi}K(p),\]
which implies that the infimum in \eqref{eq:defHdelta} is a minimum.
Moreover, because $H(x,\cdot,m)$ is convex in $p$ (Remark~\ref{rmk:cxty})  and $K$ is strictly convex, this minimum is attained at a unique point, which we denote by $q^\varepsilon(x,p,m)$.

\begin{lemma}\label{lem:inf-conv-props}
Under the assumptions of Theorem~\ref{thm:main-weak.power.growth}, let  $H^\epsi$ be the regularized Hamiltonian  in \eqref{eq:defHdelta}, and let $q^\epsi$ be the corresponding unique minimizer. For a.e.~$x \in \Tt^d$, the following properties hold:
    \begin{enumerate}[label=(\roman*)]
        \item The map $(p,m)\mapsto q^\varepsilon(x,p,m)$ is continuous.
        \item The function $q^\epsi$ admits a semi-explicit formula that ensures its measurability, and consequently the measurability of $H^\epsi$.
        \item The regularized Hamiltonian $H^\epsi$ satisfies the Caratheodory-type conditions of (\ref{a:regularityH}).
        %$H^\epsi(x,\cdot,m)$ is differentiable in $p$.
    \end{enumerate}
\end{lemma}
\begin{proof}
    (i) We argue by contradiction. Let $F(q,p,m) := H(x,p-q,m) + \frac{1}{\varepsilon}K(q)$. By assumption, $F$ is continuous in $(q,p,m)$. Let $(p_k,m_k) \to (p_0, m_0)$ as $k\to\infty$. Let $q_k := q^\varepsilon(x, p_k, m_k)$ and $q_0 := q^\varepsilon(x, p_0, m_0)$. Assume for contradiction that $q_k$ does not converge to $q_0$. This means there exists an $r>0$ and a subsequence (which we still denote by $q_k$) such that $|q_k - q_0| \geq r$ for all $k$.

    By definition, $q_k$ is the minimizer for $(p_k, m_k)$, so we have
    \begin{equation}\label{eq:minimizer_ineq}
        F(q_k, p_k, m_k) \leq F(q_0, p_k, m_k).
    \end{equation}
    The term on the right-hand side converges to $F(q_0, p_0, m_0)$ as $k\to\infty$. This implies the sequence $F(q_k, p_k, m_k)$ is bounded above. Since $F$ is coercive with respect to $q$ (i.e., $F \to \infty$ as $|q|\to\infty$), the sequence $\{q_k\}$ must be bounded.

    Since $\{q_k\}$ is bounded, there exists a further subsequence (again denoted by $q_k$) that converges to some point $q^*$. By our assumption, we must have $|q^* - q_0| \geq r$.
    Now, taking the limit of the inequality \eqref{eq:minimizer_ineq} along this subsequence, the continuity of $F$ gives:
    \[
    \lim_{k\to\infty} F(q_k, p_k, m_k) = F(q^*, p_0, m_0) \quad \text{and} \quad \lim_{k\to\infty} F(q_0, p_k, m_k) = F(q_0, p_0, m_0).
    \]
    This implies $F(q^*, p_0, m_0) \leq F(q_0, p_0, m_0)$. However, $q_0$ is the \emph{unique} minimizer of $F(\cdot, p_0, m_0)$. This forces $q^*=q_0$. This contradicts our finding that $|q^* - q_0| \geq r > 0$. Therefore, our initial assumption was false, and we must have $q_k \to q_0$. This proves continuity.
    
    \medskip
    
    (ii) Let us first denote the subdifferentials of the convex functions $K$ and $H^\epsi(x,\cdot,m)$ by $\delta K$ and $\delta_pH^\epsi(x,\cdot,m)$, respectively. Denoting $q^* := q^\epsi(x,p,m)$ for brevity, the minimization in \eqref{eq:defHdelta} gives the inclusions
    \begin{equation}\label{eq:inf-conv.diff-relation1}
        D_pH(x,p-q^*,m)\in \frac{1}{\epsi}\delta K(q^*), \qquad D_pH(x,p-q^*,m)\in \delta_p H^\epsi(x,p,m).
    \end{equation}
    
    Next, we investigate the first inclusion in \eqref{eq:inf-conv.diff-relation1} to derive a semi-explicit formula for the function $q^\epsi$. First, we re-express the inclusion as an equality in terms of the Legendre transform $\widehat K$ of $K$; that is,
    \[\widehat K(v) := \sup_q\{qv-K(q)\}.\]
    The strict convexity of $K$ ensures that $\widehat K$ is $C^1$; then, the first inclusion in \eqref{eq:inf-conv.diff-relation1} is equivalent to the equality
    \begin{equation}\label{eq:inf-conv.diff-relation2}
    q^* = D\widehat K(\epsi D_pH(x,p-q^*,m)).
    \end{equation}
    Next, a unique value of $q^*$ satisfies the first inclusion in \eqref{eq:inf-conv.diff-relation1}; hence, the map
    \[\mathbf{M}^\epsi_{x,m}\colon s\mapsto s + D\widehat K(\epsi D_pH(x,s,m))\]
    can be inverted to formally solve \eqref{eq:inf-conv.diff-relation2} as
    \begin{equation}\label{eq:q-star-semi-explicit}
        q^\epsi(x,p,m) = p - (\mathbf{M}^\epsi_{x,m})^{-1}(p).
    \end{equation}
    This formula justifies the measurability of $q^\epsi$  because the map $(x,s,m)\mapsto (x,\mathbf{M}^\epsi_{x,m}(s),m)$ is Borel measurable outside a null set of $x\in\Tt^d$. Consequently, 
\begin{equation}\label{eq:H-epsilon-formula}
    H^\epsi(x,p,m) = H(x,p-q^\epsi(x,p,m),m) + \frac{1}{\epsi}K(q^\epsi(x,p,m))
\end{equation}
is a measurable function which is continuous in $(p,m)$ as is $H$.

\medskip

    (iii) The second inclusion in \eqref{eq:inf-conv.diff-relation1} must collapse to an equality since the left-hand side is continuous in~$p$; that is,
    \begin{equation}\label{eq:DpH-epsilon-formula}
        D_pH^\epsi(x,p,m) = D_pH(x,p-q^\epsi(x,p,m),m).
    \end{equation}
    Note that this is an instance of the regularizing effect of infimal convolutions. As above, this shows that the map $(p,m)\mapsto D_pH^\epsi(x,p,m)$ is continuous; hence, $H^\epsi$ fully satisfies the Caratheodory-type conditions of (\ref{a:regularityH}).
\end{proof}

We now prove that $H^\epsi$ satisfies the monotonicity condition in Assumption~\ref{onH.monotone}.
\begin{lemma}[Monotonicity of $H^\epsi$]\label{lem:H.delta.monotone}
        Suppose that $H\colon \Tt^d\times \RR^d \times \RR^+ \to \RR$ satisfies Assumption~\ref{onH.monotone} and Assumption~\ref{onH.weakgrowth}. For $\epsi > 0$, let $H^\epsi\colon \Tt^d\times \RR^d \times \RR^+ \to \RR$ be defined as in \eqref{eq:defHdelta}. Then, the following monotonicity condition holds for a.e.~$x\in\Tt^d$ and for all $(p_1,m_1)$, $(p_2,m_2) \in\RR^d\times\RR^+$:
        \begin{equation}
    \label{h-epsilon-mon}
                \begin{aligned}
                        &\big(- H^\epsi(x,p_1,m_1) +  H^\epsi(x,p_2,m_2)\big) (m_1 - m_2)\\ &\quad + 
                        \big( m_1 D_p H^\epsi(x,p_1,m_1) - m_2 D_p H^\epsi(x,p_2,m_2)\big) \cdot (p_1-p_2) \geq 0.
                \end{aligned}
        \end{equation}
\end{lemma}
\begin{proof}
    %The discussion preceding the lemma establishes the measurability and the continuity statements, thus it is left to show \eqref{h-epsilon-mon}.
    Plugging in the formulae \eqref{eq:H-epsilon-formula} and \eqref{eq:DpH-epsilon-formula} in \eqref{h-epsilon-mon}, and denoting
    \[q_1 = q^\epsi(x,p_1,m_1), \qquad q_2 = q^\epsi(x,p_2,m_2),\]
    we see that we must show that
    \begin{equation*}
                \begin{aligned}
                        &\big(-H(x,p_1-q_1,m_1) +  H(x,p_2-q_2,m_2)\big) (m_1 - m_2)\\ &\quad + 
                        \big( m_1 D_p H(x,p_1-q_1,m_1) - m_2 D_p H(x,p_2-q_2,m_2)\big) \cdot \big((p_1-q_1)-(p_2-q_2)\big) \\
            &\quad+ m_1\Big(-D_p H(x,p_1-q_1,m_1)\cdot (q_2-q_1) - \frac{1}{\epsi}K(q_1)+\frac{1}{\epsi}K(q_2)\Big) \\
            &\quad+ m_2\Big(-D_p H(x,p_2-q_2,m_2)\cdot(q_1-q_2) - \frac{1}{\epsi}K(q_2)+\frac{1}{\epsi}K(q_1)\Big) \geq 0.
                \end{aligned}
        \end{equation*}
    The sum in the first two lines of the last expression is non-negative by Assumption~\ref{onH.monotone}, while the last two lines are each non-negative because of the convexity of $K$ and the first inclusion in~\eqref{eq:inf-conv.diff-relation1}.
\end{proof}

Next, we derive some growth estimates on $H^\epsi$ and $D_pH^\epsi$ that are useful in the sequel.
\begin{proposition}[Estimates for $H^\epsi$]\label{prop:boundsHepsilon}
    Suppose that $H$ satisfies Assumption~\ref{onH.monotone} and Assumption~\ref{onH.weakgrowth}. For $0<\epsi\leq 1$, let $H^\epsi(x,p,m)$ be defined as in \eqref{eq:defHdelta}. Then, there exist constants $C$ and $C_\epsi$ (the former being independent of $\epsi$), such that
    \begin{align}
        & H^\epsi(x,p,m)\leq C_{\epsi}|p|^\alpha-\frac{1}{C}m^\beta + C, \label{eq:H-epsilon.upper} \\
        & |D_pH^\epsi(x,p,m)| \leq C_\epsi\Big(1+|p|^{\alpha-1} + m^{\beta-\frac{\beta}{\alpha}}\Big),\label{eq:DpHdelta.upper} \\
        & H^\epsi(x,p,m) \geq \frac{1}{C}|p|^\alpha - C(m^\beta+1),\label{eq:H-epsilon.lower.uniform} \\
        & D_pH^\epsi(x,p,m)\cdot p - H^\epsi(x,p,m) \geq \frac{1}{C} m^{\beta} - C. \label{eq:DpHdelta.dotp.minusH}
        \end{align}
\end{proposition}
\begin{proof}
    Invoking \eqref{eq:defHdelta} directly with $q=p$, we get
    \[H^\epsi(x,p,m) \leq \frac{1}{\epsi}|p|^\alpha + \frac{1}{\epsi}|p| + H(x,0,m).\]
    Absorbing the middle term into the first term by Young's inequality and estimating the last term by \eqref{eq:assH.upper.weak}, we conclude \eqref{eq:H-epsilon.upper}.

       Next, recalling that $q^* = q^\epsi(x,p,m)$, we  use the conditions $K(q)\geq |q|^\alpha$ and $\epsi\leq 1$ in  formula \eqref{eq:H-epsilon-formula} and use \eqref{eq:assH.lower.weak} to estimate the other term to obtain
    \begin{equation}\label{H-inf-conv-lower}
        H^\epsi(x,p,m) \geq |q^*|^\alpha + \frac{1}{C}|p-q^*|^{\alpha} - C(m^\beta+1).
    \end{equation}
        This estimate is key to the conclusion of both \eqref{eq:DpHdelta.upper} and \eqref{eq:H-epsilon.lower.uniform}.
    
    First, combining \eqref{H-inf-conv-lower} with \eqref{eq:H-epsilon.upper} and rearranging, we find that
    \[|q^\epsi(x,p,m)|^\alpha \leq C_\epsi|p|^\alpha + Cm^\beta + C.\]
    Hence,
    \begin{equation}\label{eq:q-epsilon-estimate}
        |q^\epsi(x,p,m)|\leq C_\epsi|p| + Cm^{\frac{\beta}{\alpha}} + C.
    \end{equation}
    On the other hand, \eqref{eq:DpH-epsilon-formula} together with the first inclusion in \eqref{eq:inf-conv.diff-relation1} provides the estimate
    \[|D_pH^\epsi(x,p,m)| \leq \frac{1}{\epsi}(1+\alpha|q^\epsi(x,p,m)|^{\alpha-1}).\]
    Thus, using \eqref{eq:q-epsilon-estimate} in the last estimate, we conclude \eqref{eq:DpHdelta.upper}.

    Secondly, noting that $\max\{|q^*|,|p-q^*|\}\geq |p|/2$, we have 
    \[\max\left\{|q^*|^\alpha, \frac{1}{C}|p-q^*|^{\alpha}\right\} \geq \frac{1}{2^\alpha(C+1)}|p|^\alpha.\]%\min\left\{\frac{1}{2^\rho}|p|^\rho, \frac{1}{C2^\alpha}\left(\frac{|p|^\alpha}{m^\tau+1}\right)\right\}
    Then, plugging this observation into \eqref{H-inf-conv-lower}, we conclude \eqref{eq:H-epsilon.lower.uniform}.

    Finally, we prove \eqref{eq:DpHdelta.dotp.minusH} as in Lemma~\ref{lem:A1+anyother}. From Lemma~\ref{lem:H.delta.monotone}, we get that $H^\epsi(x,\cdot,m)$ is convex as in Remark~\ref{rmk:cxty}; hence, we have
    \[D_pH^\epsi(x,p,m)\cdot p - H^\epsi(x,p,m) \geq -H^\epsi(x,0,m).\]
    Combining this with \eqref{eq:H-epsilon.upper}, we conclude \eqref{eq:DpHdelta.dotp.minusH}.
\end{proof}

Note that Lemma~\ref{lem:H.delta.monotone} and Proposition~\ref{prop:boundsHepsilon} ensure the existence of the extended functions $H^\epsi\colon \Tt^d\times \RR^d \times \RR^+_0 \to \RR$ and $mD_pH^\epsi\colon \Tt^d\times \RR^d \times \RR^+_0 \to \RR^d$ as in (\ref{a:Hat0}) and (\ref{a:mDpHat0}). Moreover, $(p,m)\mapsto H^\epsi(x,p,m)$ is continuous on $\RR^d\times \RR^+_0$ as in the proof of Theorem~\ref{thm:main-simple.power.growth}.

Now, the bounds in Proposition~\ref{prop:boundsHepsilon} as well as the fact that $V\in L^{\bar\beta'}$, allow us to define, for any $0 < \epsi \leq 1$, a regularized operator
\[A_\epsi: L^{\bar\beta}(\Tt^d;\RR^+_0) \times W^{1,\bar\gamma}(\Tt^d) \to L^{\bar\beta'}(\Tt^d) \times W^{-1,\bar\gamma'}(\Tt^d)\]
that maps an element \((m,u) \in L^{\bar\beta}(\Tt^d;\RR^+_0) \times W^{1,\bar\gamma}(\Tt^d) \) to an element \(A_\epsi[m,u]\in L^{\bar\beta'}(\Tt^d) \times W^{-1,\bar\gamma'}(\Tt^d) \) defined for \((\eta, \nu)\in L^{\bar\beta}(\Tt^d) \times W^{1,\bar\gamma}(\Tt^d)\) by
\begin{equation}\label{defAepsi.delta}
        \begin{aligned}
                \langle A_\epsi[m,u], (\eta, \nu) \rangle :=& \int_{\Tt^d} \bigg( -u -H^\epsi(x,D u,m) + V(x) \bigg)\eta \,
dx \\
                &\enspace+\int_{\Tt^d}  \Big(m D_pH^\epsi(x, D u,m) \cdot D \nu  + (m-1)\nu \Big) \, dx \\
        &\quad +\int_{\Tt^d} \Big(\epsi|D u|^{\bar\gamma-2} D u \cdot D \nu 
                +\epsi |u|^{\bar\gamma-2}u\nu\Big) \, dx.
        \end{aligned}
\end{equation}
Below, we apply Theorem~\ref{thm:monotone.abstract-weak} to this operator as in the previous sections.

\begin{proposition}[Existence of a Regularized Solution]\label{thm:strongAepsi.delta}
Under the assumptions of Theorem~\ref{thm:main-weak.power.growth}, for $0 < \epsi\leq 1$, let $A_\epsi$ be the operator in \eqref{defAepsi.delta}. Then, there exists \(( m_\epsi, u_\epsi) \in L^{\bar\beta}(\Tt^d;\RR^+_0) \times W^{1,\bar\gamma}(\Tt^d)\) satisfying for all \(( \mu, \upsilon) \in L^{\bar\beta}(\Tt^d;\RR^+_0) \times W^{1,\bar\gamma}(\Tt^d)\) the condition
        \begin{equation*}
                \langle A_\epsi[m_\epsi,  u_\epsi],
                (\mu,  \upsilon) - (m_\epsi,  u_\epsi) \rangle_{L^{\bar\beta'}(\Tt^d;\RR^+_0) \times W^{-1,\bar\gamma'}(\Tt^d), L^{\bar\beta}(\Tt^d;\RR^+_0) \times W^{1,\bar\gamma}(\Tt^d)} \geq 0.
        \end{equation*}
\end{proposition}

\begin{proof}
As in the proofs of Propositions~\ref{thm:strongAepsi} and~\ref{thm:weakAepsi.congestion}, we first see that $A_\epsi$ is a monotone operator, using Lemma~\ref{lem:H.delta.monotone} in place of Assumption~\ref{onH.monotone}. Then, we show the coercivity of $A_\epsi$ along the same lines as Proposition~\ref{prop:coercAepsi}. More precisely, we get the equality \eqref{eq:coercive-test.Aepsi} with $H$ replaced by $H^\epsi$, and we obtain the estimate in \eqref{eq:coercive.Aepsi} by using \eqref{eq:DpHdelta.dotp.minusH} in place of \eqref{eq:DpHdotpminusH-estimate1}, \eqref{eq:H-epsilon.lower.uniform} in place of \eqref{eq:assH.lower.simple}, and other bounds similarly. Consequently, Theorem~\ref{thm:monotone.abstract-weak} applies to $A_\epsi$ and there exists \(( m_\epsi,  u_\epsi) \in L^{\bar\beta}(\Tt^d;\RR^+_0) \times W^{1,\bar\gamma}(\Tt^d)\) which satisfies 
\begin{equation*}%\label{eq:Aepsi.delta.weak.ineq}
                \langle A_\epsi[\mu, \upsilon],
                (\mu,  \upsilon) - (m_\epsi,  u_\epsi) \rangle_{L^{\bar\beta'}(\Tt^d;\RR^+_0) \times W^{-1,\bar\gamma'}(\Tt^d), L^{\bar\beta}(\Tt^d;\RR^+_0) \times W^{1,\bar\gamma}(\Tt^d)} \geq 0
\end{equation*}
    for all \(( \mu,  \upsilon) \in L^{\bar\beta}(\Tt^d;\RR^+_0) \times W^{1,\bar\gamma}(\Tt^d)\). Finally, the hemicontinuity of $A_\epsi$ follows as in the proof of Proposition~\ref{thm:strongAepsi},
    thus we conclude by Remark~\ref{rmk:hemicontinuous-upgrade}.
\end{proof}

We now observe that $(m_\epsi, u_\epsi)$ provided by Proposition~\ref{thm:strongAepsi.delta} satisfies the strong solution property.

\begin{lemma}[Strong solution property]\label{lem:strong.is.strong.delta}
    Let $0<\epsi\leq 1$ and  $(m, u)\in L^{\bar\beta}(\Tt^d;\RR^+_0) \times W^{1,\bar\gamma}(\Tt^d)$ be such that
\begin{equation*}%\label{eq:Aepsi.delta.strong-ineq}
    \langle A_\epsi[m,u], (\mu-m, \upsilon-u) \rangle \geq 0
\end{equation*}
    for all $(\mu, \upsilon)\in L^{\bar\beta}(\Tt^d;\RR^+_0) \times W^{1,\bar\gamma}(\Tt^d)$.
    Then, we have
    \begin{align*}
            & u + H^\epsi(x,Du,m)- V(x) \leq 0 \qquad \text{a.e.~in~} \Tt^d,%\label{eq:Hdelta-leq0}
\\
            & u+ H^\epsi(x,Du,m)- V(x) = 0 \qquad \text{a.e.~in~} \{x\in\Tt^d\colon m>0\}, %\label{eq:Hdelta-0}
    \end{align*}
    and
    \begin{equation*}%\label{eq:delta-transport}
        \int_{\Tt^d} \big[ (m-1)\varphi + mD_pH^\epsi(x,Du,m)\cdot D\varphi + \epsi|Du|^{\bar\gamma-2}Du\cdot D\varphi + \epsi |u|^{\bar\gamma-2}u\varphi\big]\, dx = 0
    \end{equation*}
    for all $\varphi\in W^{1,\bar\gamma}(\Tt^d)$.
\end{lemma}

\begin{proof}
The proof is the same as that of Lemma~\ref{lem:strong.is.strong} verbatim, with $H$ replaced by $H^\epsi$ everywhere.
\end{proof}

\begin{theorem}[Uniform estimates for \((m_\epsi,u_\epsi)\)]\label{thm:apriori-weak.power.growth}
        Under the assumptions of Theorem~\ref{thm:main-weak.power.growth}, let \(( m_\epsi,  u_\epsi) \in L^{\bar\beta}(\Tt^d;\RR^+_0) \times W^{1,\bar\gamma}(\Tt^d)\) be given by Proposition~\ref{thm:strongAepsi.delta}. Then,  
        \begin{equation*}
                %\label{eq:apriori-weak.power.growth}
                \begin{aligned}
                        & \Vert m_\epsi\Vert_{L^{\bar\beta}(\Tt^d)}^{\bar\beta} + \Vert u_\epsi \Vert_{W^{1,\bar\gamma}(\Tt^d)}^{\bar\gamma}  \leq C
                \end{aligned}
        \end{equation*}
    for a constant $C>0$ as in Remark~\ref{rmk:onC}; in particular, it is independent of $\epsi$.
\end{theorem}

\begin{proof} We follow the proof of Theorem~\ref{thm:apriori-simple.power.growth} line by line, only replacing Lemma~\ref{lem:strong.is.strong} by Lemma~\ref{lem:strong.is.strong.delta}, \eqref{eq:DpHdotpminusH-estimate1} by \eqref{eq:DpHdelta.dotp.minusH}, and \eqref{eq:assH.lower.simple} by \eqref{eq:H-epsilon.lower.uniform}.
\end{proof}

The crucial observation for passing to the limit is the following remark.

%%{\color{ForestGreen}
%%\begin{remark}
%%\label{observation}
%%We note that since $\alpha>1$, $\delta K(0)= B_1(0)$.
%%Fix $M>0$. Suppose that 
%%$x,p,m$ are such that
%%\[
%%|D_pH(x,p,m)|<M
%%\]
%%and let $q^\epsi$ be the minimizer on the definition of $H^\epsi(x,p,m)$.
%%Then, there is $\epsi_0$ depending only on $M$ such that for all 
%%$\epsi<\epsi_0$, the inclusion
%%\[
%%D_pH(x,p,m)\in \frac{1}{\epsi}\delta K(0)
%%\]
%%implies $q^\epsi=0$, by uniqueness of $q^\epsi$ and the inclusion
%%\[
%%D_pH(x,p,m)\in \frac{1}{\epsi}\delta K(q^\epsi).
%%\]
%%Accordingly,  $H^\epsi(x,p,m)=H(x,p,m)$ for all $\epsi<\epsi_0$.
%%\end{remark}}

\begin{remark}
\label{observation}
We note that since $\alpha>1$, $\delta K(0)= B_1(0)$.
Fix $M>0$, and  suppose that 
$x,p,m$ are such that
\[
|D_pH(x,p,m)|<M.
\]
Then, there is $\epsi_0$ depending only on $M$ such that for all 
$\epsi<\epsi_0$, the inclusion
\[
D_pH(x,p,m)\in \frac{1}{\epsi}\delta K(0)
\]
holds. Recall that $q^\epsi(x,p,m)$ is the minimizer on the definition of $H^\epsi(x,p,m)$ and it is the unique value satisfying the inclusion
\[
D_pH(x,p-q^\epsi(x,p,m),m)\in \frac{1}{\epsi}\delta K(q^\epsi(x,p,m))
\]
as given in~\eqref{eq:inf-conv.diff-relation1}. This implies that $q^\epsi(x,p,m) = 0$ and accordingly,  $H^\epsi(x,p,m) = H(x,p,m)$ and $D_pH^\epsi(x,p,m) = D_pH(x,p,m)$ for all $\epsi<\epsi_0$.
\end{remark}

\begin{proof}[Proof of Theorem~\ref{thm:main-weak.power.growth}]

We begin by recalling from Proposition~\ref{thm:strongAepsi.delta} that for any sequence $\epsi_k \to 0^+$, we have a corresponding sequence of solutions $(m_{\epsi_k}, u_{\epsi_k})_k \subset L^{\bar\beta}(\Tt^d;\RR^+_0) \times W^{1,\bar\gamma}(\Tt^d)$
satisfying
\begin{equation*}
%\label{eq:existweaksol.delta.regularized}
\begin{aligned}
0 &\leq \langle A_{\epsi_k}  [\mu,
\upsilon] , (\mu, \upsilon)-(m_{\epsi_k},
u_{\epsi_k})\rangle \\
&= \int_{\Tt^d} \Big(  -\upsilon- H^\epsi(x, D\upsilon,\mu) + V(x)  \Big)(\mu \, -m_{\epsi_k})\,dx
\\ &\quad+\int_{\Tt^d} \Big( \mu D_p H^\epsi(x, D\upsilon,\mu) \cdot(D\upsilon - Du_{\epsi_k})  + (\mu-1)( \upsilon-u_{\epsi_k}) \Big)
\, dx
\\&\quad+\epsi_k\int_{\Tt^d}
 \Big( |D\upsilon|^{\bar\gamma-2} D\upsilon \cdot D( \upsilon-u_{\epsi_k}) 
+|\upsilon|^{\bar\gamma-2}\upsilon( \upsilon-u_{\epsi_k})\Big)\,dx
\end{aligned}
\end{equation*}
for every \((\mu, \upsilon)\in L^{\bar\beta}(\Tt^d;\RR^+_0) \times W^{1,\bar\gamma}(\Tt^d) \) and \(k\in\Nn\). Fix \((\mu, \upsilon)\in L^{\infty}(\Tt^d) \times W^{1,\infty}(\Tt^d) \) with $\essinf\mu > 0$ as in Definition~\ref{def:weak}. By \eqref{eq:assH.DpH.upper.weak}, we get $D_pH(x,D\upsilon(x), \mu(x)) \in L^\infty(\Tt^d)$. Hence, by Remark \ref{observation}, for sufficiently small $\epsi>0$, we have
\[H^\epsi(x,D\upsilon,\mu) = H(x,D\upsilon,\mu), \qquad D_pH^\epsi(x,D\upsilon,\mu) = D_pH(x,D\upsilon,\mu).\]
Therefore, for every \((\mu, \upsilon)\in L^{\infty}(\Tt^d) \times W^{1,\infty}(\Tt^d) \) with $\essinf\mu > 0$, we have
\begin{equation}
\label{eq:existweaksol.delta}
\begin{aligned}
0 &\leq \int_{\Tt^d} \Big(  -\upsilon- H(x, D\upsilon,\mu) + V(x)  \Big)(\mu \, -m_{\epsi_k})\,dx
\\ &\quad+\int_{\Tt^d} \Big( \mu D_p H(x, D\upsilon,\mu) \cdot(D\upsilon - Du_{\epsi_k})  + (\mu-1)( \upsilon-u_{\epsi_k}) \Big)
\,dx
\\&\quad+\epsi_k\int_{\Tt^d}
 \Big( |D\upsilon|^{\bar\gamma-2} D\upsilon \cdot D( \upsilon-u_{\epsi_k}) +|\upsilon|^{\bar\gamma-2}\upsilon( \upsilon-u_{\epsi_k})\Big)\,dx
\end{aligned}
\end{equation}
for sufficiently large $k \in \Nn$.

On the other hand, by Theorem~\ref{thm:apriori-weak.power.growth}, we can find \((m,u)\in L^{\bar\beta}(\Tt^d;\RR^+_0)\times W^{1,\bar\gamma}(\Tt^d)\) such that, up to a subsequence that we do not relabel,
\begin{equation}
\label{eq:compactness.delta}
\begin{aligned}
& (m_{\epsi_k}, u_{\epsi_k}) \rightharpoonup (m,u)\quad \text{weakly in } L^{\bar\beta}(\Tt^d;\RR^+_0)\times W^{1,\bar\gamma}(\Tt^d).
\end{aligned}
\end{equation}
We conclude the proof by letting \(k\to\infty\) in \eqref{eq:existweaksol.delta} and using \eqref{eq:compactness.delta} combined with H\"older's inequality.
\end{proof}

\begin{remark}
    We note that, under the conditions of Theorem~\ref{thm:main-weak.power.growth}, we can only guarantee the existence of a weak solution because, unlike the power-growth and congestion cases, Assumption~\ref{onH.weakgrowth} does not provide a quantitative upper bound on $|D_pH(x,p,m)|$ in terms of $|p|$. Hence, we lack the necessary uniform integrability for the sequence $m_{\epsi_k} D_pH(x,Du_{\epsi_k},m_{\epsi_k})$. This prevents us from passing to the limit in the transport equation to obtain strong solutions (as was done via Lemma~\ref{lem:weak.is.strong}), forcing us to rely on the weak variational formulation with highly regular test functions.
\end{remark}

\bibliographystyle{abbrv}

\begin{thebibliography}{10}

\bibitem{ADFGU}
A.~Alharbi, D.~Gomes, G.~D. Fazio, and M.~Ucer.
\newblock Regularity for weak solutions to first-order local mean field games.
\newblock 11 2024.

\bibitem{almulla2017two}
N.~Almulla, R.~Ferreira, and D.~Gomes.
\newblock Two numerical approaches to stationary mean-field games.
\newblock {\em Dynamic Games and Applications}, 7(4):657--682, 2017.

\bibitem{Ambrose}
D.~M. Ambrose.
\newblock Small strong solutions for time-dependent mean field games with local
  coupling.
\newblock {\em Comptes Rendus Mathematique. Academie des Sciences. Paris},
  354(6):589--594, 2016.

\bibitem{bertucciMonotoneSolutionsMean2023}
C.~Bertucci.
\newblock Monotone solutions for mean field games master equations: Continuous
  state space and common noise.
\newblock {\em Communications in Partial Differential Equations},
  48(10-12):1245--1285, 2023.

\bibitem{bocorsporr}
L.~Boccardo, L.~Orsina, and A.~Porretta.
\newblock Strongly coupled elliptic equations related to mean-field games
  systems.
\newblock {\em Journal of Differential Equations}, 261(3):1796--1834, 2016.

\bibitem{cgbt}
P.~Cardaliaguet, P.~Garber, A.~Porretta, and D.~Tonon.
\newblock Second order mean field games with degenerate diffusion and local
  coupling.
\newblock {\em NoDEA. Nonlinear Differential Equations and Applications},
  22(5):1287--1317, 2015.

\bibitem{cardaliaguetMeanFieldGames2014}
P.~Cardaliaguet and P.~Graber.
\newblock Mean field games systems of first order.
\newblock {\em ESAIM: Control, Optimisation and Calculus of Variations}, 21,
  Jan. 2014.

\bibitem{cirant3}
M.~Cirant.
\newblock Stationary focusing mean-field games.
\newblock {\em Communications in Partial Differential Equations},
  41(8):1324--1346, 2016.

\bibitem{MR4064664}
M.~Cirant, R.~Gianni, and P.~Mannucci.
\newblock Short-time existence for a general backward-forward parabolic system
  arising from mean-field games.
\newblock {\em Dynamic Games and Applications}, 10(1):100--119, 2020.

\bibitem{CirGo20}
M.~Cirant and A.~Goffi.
\newblock Maximal {{L}}{\textsuperscript{q}}-regularity for parabolic
  {{Hamilton-Jacobi}} equations and applications to mean field games.
\newblock {\em Annals of PDE. Journal Dedicated to the Analysis of Problems
  from Physical Sciences}, 7(2):Paper No. 19, 40, 2021.

\bibitem{ciranttonon}
M.~Cirant and D.~Tonon.
\newblock Time-dependent focusing mean-field games: The sub-critical case.
\newblock {\em Journal of Dynamics and Differential Equations}, 31(1):49--79,
  2019.

\bibitem{debrunner-flor}
H.~Debrunner and P.~Flor.
\newblock Ein {E}rweiterungssatz f\"ur monotone {M}engen.
\newblock {\em Arch. Math.}, 15:445--447, 1964.

\bibitem{FG2}
R.~Ferreira and D.~Gomes.
\newblock Existence of weak solutions to stationary mean-field games through
  variational inequalities.
\newblock {\em Siam Journal On Mathematical Analysis}, 50(6):5969--6006, 2018.

\bibitem{FGT1}
R.~Ferreira, D.~Gomes, and T.~Tada.
\newblock Existence of weak solutions to first-order stationary mean-field
  games with {{Dirichlet}} conditions.
\newblock {\em Proceedings of the American Mathematical Society},
  147(11):4713--4731, 2019.

\bibitem{FeGoTa21}
R.~Ferreira, D.~Gomes, and T.~Tada.
\newblock Existence of weak solutions to time-dependent mean-field games.
\newblock {\em Nonlinear Analysis}, 212:Paper No. 112470, 31, 2021.

\bibitem{FGTchapter2026}
R.~Ferreira, D.~Gomes, and T.~Tada.
\newblock {\em An Introduction to Monotonicity Methods in Mean-Field Games},
  pages 91--132.
\newblock Springer Nature Switzerland, Cham, 2026.

\bibitem{MR4175148}
D.~Gomes, H.~Mitake, and K.~Terai.
\newblock The selection problem for some first-order stationary mean-field
  games.
\newblock {\em Networks and Heterogeneous Media}, 15(4):681--710, 2020.

\bibitem{GPatVrt}
D.~Gomes, S.~Patrizi, and V.~Voskanyan.
\newblock On the existence of classical solutions for stationary extended mean
  field games.
\newblock {\em Nonlinear Analysis}, 99:49--79, 2014.

\bibitem{Gomes2015b}
D.~Gomes and E.~Pimentel.
\newblock Time dependent mean-field games with logarithmic nonlinearities.
\newblock {\em Siam Journal On Mathematical Analysis}, 47(5):3798--3812, 2015.

\bibitem{Gomes2016c}
D.~Gomes and E.~Pimentel.
\newblock Local regularity for mean-field games in the whole space.
\newblock {\em Minimax Theory and its Applications}, 01(1):065--082, 2016.

\bibitem{GPM2}
D.~Gomes, E.~Pimentel, and H.~{S{\'a}nchez-Morgado}.
\newblock Time-dependent mean-field games in the subquadratic case.
\newblock {\em Communications in Partial Differential Equations}, 40(1):40--76,
  2015.

\bibitem{GPM3}
D.~Gomes, E.~Pimentel, and H.~{S{\'a}nchez-Morgado}.
\newblock Time-dependent mean-field games in the superquadratic case.
\newblock {\em ESAIM. Control, Optimisation and Calculus of Variations},
  22(2):562--580, 2016.

\bibitem{GM}
D.~Gomes and H.~S{\'a}nchez~Morgado.
\newblock A stochastic {{Evans-Aronsson}} problem.
\newblock {\em Transactions of the American Mathematical Society},
  366(2):903--929, 2014.

\bibitem{GJS2}
D.~Gomes and J.~Sa{\'u}de.
\newblock Numerical methods for finite-state mean-field games satisfying a
  monotonicity condition.
\newblock {\em Applied Mathematics \& Optimization}, 2018.

\bibitem{Caines2}
M.~Huang, P.~E. Caines, and R.~P. Malham{\'e}.
\newblock Large-population cost-coupled {{LQG}} problems with nonuniform
  agents: Individual-mass behavior and decentralized {{$\epsilon$}}-{{Nash}}
  equilibria.
\newblock {\em Institute of Electrical and Electronics Engineers},
  52(9):1560--1571, 2007.

\bibitem{KiSt00}
D.~Kinderlehrer and G.~Stampacchia.
\newblock {\em An Introduction to Variational Inequalities and Their
  Applications}, volume~31 of {\em Classics in Applied Mathematics}.
\newblock {Society for Industrial and Applied Mathematics (SIAM), Philadelphia,
  PA}, 2000.

\bibitem{ll1}
J.-M. Lasry and P.-L. Lions.
\newblock Jeux {\`a} champ moyen. {{I}}. {{Le}} cas stationnaire.
\newblock {\em Comptes Rendus Mathematique. Academie des Sciences. Paris},
  343(9):619--625, 2006.

\bibitem{ll2}
J.-M. Lasry and P.-L. Lions.
\newblock {Jeux {\`a} champ moyen. II -- Horizon fini et contr{\^o}le optimal}.
\newblock {\em Comptes Rendus. Math{\'e}matique}, 343(10):679--684, 2006.

\bibitem{lasryMeanFieldGames2007a}
J.-M. Lasry and P.-L. Lions.
\newblock Mean field games.
\newblock {\em Japanese Journal of Mathematics}, 2(1):229--260, Mar. 2007.

\bibitem{AMFS}
A.~M{\'e}sz{\'a}ros and F.~Silva.
\newblock On the {{Variational Formulation}} of {{Some Stationary Second-Order
  Mean Field Games Systems}}.
\newblock {\em Siam Journal On Mathematical Analysis}, 50(1):1255--1277, 2018.

\bibitem{nurbekyanNoteConvergenceMonotone2023}
L.~Nurbekyan.
\newblock A note on the convergence of the monotone inclusion version of the
  primal-dual hybrid gradient algorithm, Nov. 2023.

\bibitem{nurbekyanMonotoneInclusionMethods2024}
L.~Nurbekyan, S.~Liu, and Y.~T. Chow.
\newblock Monotone inclusion methods for a class of second-order non-potential
  mean-field games, Mar. 2024.

\bibitem{PV15}
E.~Pimentel and V.~Voskanyan.
\newblock Regularity for second-order stationary mean-field games.
\newblock {\em Indiana University Mathematics Journal}, 66(1):1--22, 2017.

\bibitem{porretta}
A.~Porretta.
\newblock On the planning problem for the mean field games system.
\newblock {\em Dynamic Games and Applications}, 4(2):231--256, 2014.

\bibitem{porretta2}
A.~Porretta.
\newblock Weak solutions to {{Fokker-Planck}} equations and mean field games.
\newblock {\em Archive for Rational Mechanics and Analysis}, 216(1):1--62,
  2015.

\bibitem{MR3691806}
A.~Porretta.
\newblock On the weak theory for mean field games systems.
\newblock {\em Bollettino dell'Unione Matematica Italiana}, 10(3):411--439,
  2017.

\bibitem{MR3113415}
V.~K. Voskanyan.
\newblock Some estimates for stationary extended mean field games.
\newblock {\em Doklady. Natsionalnaya Akademiya Nauk Armenii. Reports. National
  Academy of Sciences of Armenia}, 113(1):30--36, 2013.

\end{thebibliography}

\end{document}